\documentclass[12pt]{amsart}
\usepackage{geometry}                
\usepackage{graphicx}
\usepackage{amssymb}
\usepackage{epstopdf}
\usepackage{hyperref}
\usepackage{ascmac}
\usepackage{amsmath}
\usepackage[all]{xy}
\usepackage[dvipsnames]{xcolor}
\usepackage{enumerate}
\usepackage{amsthm}
\usepackage{appendix}
\usepackage{mathrsfs}
\usepackage{bm}
\usepackage{comment}
\usepackage{longtable}
\usepackage{float}
\usepackage{rotating}
\usepackage{graphicx}
\usepackage{pdflscape}
\usepackage{here}

\numberwithin{equation}{section}

\DeclareMathOperator{\CC}{\mathbb{C}}
\DeclareMathOperator{\RR}{\mathbb{R}}

\DeclareMathOperator{\ZZ}{\mathbb{Z}}

\DeclareMathOperator{\diag}{diag}

\DeclareMathOperator{\ii}{\sqrt{-1}}

\DeclareMathOperator{\Aut}{Aut}

\DeclareMathOperator{\tr}{tr}
\DeclareMathOperator{\Br}{Br}
\DeclareMathOperator{\Real}{Re}
\DeclareMathOperator{\ord}{order}
\DeclareMathOperator{\per}{period}

\newcommand{\tilS}{\tilde{S}^{(0)}}
\newcommand{\tilM}{\tilde{M}^{(0)}}

\newcommand{\tilQ}{\tilde{Q}^{(0)}}

\newcommand{\uppS}{\tilde{S}^{\Delta}}
\newcommand{\ee}{\underline{e}}

\theoremstyle{definition}
\newtheorem{theo}{Theorem}[section]
\newtheorem{prop}[theo]{Proposition}
\newtheorem{lem}[theo]{Lemma}
\newtheorem{cor}[theo]{Corollary}
\newtheorem{defi}[theo]{Definition}
\newtheorem{rem}[theo]{Remark}
\newtheorem{exa}[theo]{Example}

\theoremstyle{theorem}

\title{Distinguished bases and the braid group action for solutions of the tt*-Toda equations}
\author{Yudai Hateruma}
\begin{document}
\maketitle

\begin{abstract}
We implement (in a concrete example) a procedure of Dubrovin which gives a braid group action on spaces of solutions of ordinary differential equations such as the quantum differential equation of complex projective space. Such actions originated in physics for the tt* (topological-antitopological fusion) equations. We make the procedure systematic by utilizing an algebraic description of braid group actions on distinguished bases. We also generalize to the case of (arbitrary) solutions of the tt*-Toda equations, one of which corresponds to the quantum cohomology of complex projective space.
\end{abstract}


\section{Introduction}

In this article we investigate an action of the braid group on the Stokes data of the tt* (topological-antitopological fusion) equations. The tt* equations — a system of nonlinear partial differential equations — were introduced by Cecotti and Vafa \cite{CV1} as equations for massive deformations of supersymmetric quantum field theories, and it was demonstrated by Dubrovin \cite{Du1} that such deformations (i.e. solutions of the tt* equations) correspond to certain monodromy data of an associated linear ordinary differential equation. This data includes Stokes matrices as the associated ODE is meromorphic on the Riemann sphere with two poles of order 2.

On the basis of physical arguments, Cecotti and Vafa conjectured that ``physical" solutions of the tt* equations correspond to Stokes matrices with integer entries, and furthermore that equivalence classes of Stokes matrices under an action of the braid group correspond to the same theory. They proposed this \cite{CV2} as an approach to the classification of certain field theories, giving mathematical evidence mainly in the case of $2 \times 2$ and $3 \times 3$ matrices. Subsequently, a remarkable analogy with a braid group action in the theory of helices in algebraic geometry emerged \cite{Za}.

These topics have been investigated by mathematicians, with significant progress in special situations. In a series of works \cite{CDG1}, \cite{CDG2}, \cite{CDG3}, Cotti-Dubrovin-Guzzetti studied the quantum cohomology of $\mathbb{C}P^n$ (complex projective space) and also Grassmannians from the point of view of the braid group action on Stokes matrices of quantum differential equations, building on previous work by Guzzetti \cite{Gu}, and Dubrovin \cite{Du2}. The quantum differential equations here are meromorphic on the Riemann sphere with a pole of order 2 and a pole of order 1. Although these works did not consider the tt* equations, a direct connection with the tt* equations was given by Guest-Its-Lin \cite{GIL4}, who solved the tt*-Toda equations (tt* equations of Toda-type) and showed that their Stokes matrices coincide with those of an associated system of linear ordinary differential equations. For a particular solution of the tt*-Toda equations this ODE is exactly the quantum differential equation of $\mathbb{C}P^n$. For arbitrary solutions whose Stokes matrices have integer entries (i.e. the physical solutions, as conjectured by Cecotti and Vafa), there is a corresponding linear equation with the same Stokes matrices. Thus a considerable amount of information is available concerning solutions of the tt*-Toda equations and their Stokes matrices.

Our purpose is to explain in detail the braid group action in this situation. As an illustrative example, we follow the strategy which was outlined in Example 4.4 of \cite{Du2} for the quantum cohomology of $\mathbb{C}P^2$, but in the more general context of arbitrary solutions of the tt*-Toda equations; this includes all solutions whose Stokes matrices have integer entries (one of which corresponds to the quantum cohomology of $\mathbb{C}P^2$).

We make the procedure systematic by utilizing an algebraic description of braid group actions on distinguished bases, for which our main reference is the monograph of Hertling-Larabi \cite{HL}. Our treatment involves
the following two aspects which are not covered in the existing literature. First we set up and investigate the data of a unimodular bilinear lattice, not as an abstract algebraic object, but concretely in function-theoretic terms. Second we work in the general context of the tt*-Toda equations (tt* equations of Toda type), which allows us to extend the range of applicability beyond the well-known special cases corresponding to quantum cohomology \cite{CDG1}, \cite{CDG2}, \cite{CDG3} or unfoldings of singularities \cite{HL}.

The paper is organised as follows. In section 2 we prepare some background material on the tt*-Toda equations, following \cite{GIL4}. In section 3 we define the (abstract) braid group action, following \cite{HL}. Our main results are presented in section 4. These can be summarized as follows.

In section \ref{subsec:A bilinear vector space} we define an $(n+1)$-dimensional vector space $V$ (the space of solutions of a certain linear ODE) and a bilinear form $L$ (Definition \ref{the-bilinear-form}). Using a cyclic symmetry property of the ODE, we introduce an automorphism $\phi$ of $V$ preserving the bilinear form $L$.
We discuss two avatars of $\phi$: a matrix $M_F$ (which can be interpreted as an $(n+1)$-th root of the monodromy), and a recursion for solutions of the ODE (which is function-theoretic, with no analogue in the algebraic theory). In section \ref{subsec:Characterization} we exhibit a collection of ``triangular bases", associated to a solution of the tt*-Toda equations. These are related to each other by the action of elements of the braid group $\Br_{n+1}$. We also give a characterization of $L$ in these terms. In section \ref{subsec:Periodicity} we discuss the action of the element $\sigma_1  \cdots  \sigma_n$ of $\Br_{n+1}$, and show that it corresponds to the automorphism $\phi$. We prove (Theorem \ref{periodicity-of-the-braid}) that it has finite order (and we compute this order) when the holomorphic data of the solution is integral. In section \ref{3times3-example} we give some explicit results in the case $n=2$. We define distinguished bases, and show that the triangular bases of section \ref{subsec:Characterization} are a (proper) subset. Finally, in section \ref{subsec:A closed formula for L}, we give a closed formula for $L$, and show that it is equivalent to a sequence of real numbers $(a_r)_{r \in \ZZ}$. The finiteness of orbits of (subgroups of) the braid group action is of particular interest, so we summarise (in Table \ref{tab:sequence-signature-orbits}) some properties which are known in the integer Stokes cases, and relate these to properties of $L$ and $(a_r)_{r \in \ZZ}$.

\section{The tt*-Toda equations}

We review the theory of the tt*-Toda equations.
Our main references are the series of papers by Guest-Its-Lin and Guest-Ho (e.g. \cite{GIL4}, \cite{GH1}).

\begin{defi}
Let $t$ be the coordinates of $\CC^*=\CC \setminus \{0\} $.
The tt*-Toda equations are the two-dimensional periodic Toda lattice with signs of coefficients chosen as follows,
\begin{align*} \label{GIL1-1.1}
	2(w_i)_{t\bar{t}} = - e^{2(w_{i+1} - w_i)} + e^{2(w_i - w_{i-1})}, \,\,\, w_i : \CC^* \to \RR, \,\,\, i \in \ZZ
\end{align*}
where  
$w_i = w_{i+n+1}$, and $w_i$ satisfy the following conditions:
\begin{itemize}
	\item $w_i + w_{n-i} = 0$ (anti-symmetry condition).
	\item $w_i=w_i(x)$, $x = |t|$ (radial condition). 
\end{itemize}
\end{defi}

\begin{theo} [\cite{GIL4}, \cite{Mo2}] \label{asymptotic-data}
Let $\mathcal{A}_{n+1}$ be the following convex polytope in $\RR^{n+1}$:
\begin{align*}
	\mathcal{A}_{n+1}=\{ m = (m_0,\cdots, m_n) \in \RR^{n+1}  \,\, | \,\, &m_{i-1} - m_i+1 \geq 0, m_i + m_{n-i } = 0  \}.
\end{align*} 
For any $m \in \mathcal{A}_{n+1}$, there exists a unique global solution $w = \diag(w_0, \cdots, w_n)$ of the tt*-Toda equations such that
\begin{equation*}
	w_i \sim - m_i \log{x},\,\, \text{as} \,\, x \to 0. \,\,\,\,\, 
\end{equation*}
Moreover, all global solutions of the tt*-Toda equations are of this form.
We call $m_i$ the asymptotic data.
\end{theo}

Dubrovin showed that the tt* equations have a zero-curvature formulation and an isomonodromy formulation.
In particular the tt*-Toda equations have such formulations.
We introduce the associated connection forms on the trivial vector bundle $\CC^{n+1} \times \CC^* \to \CC^*$ following \cite{GIL4}.

The tt*-Toda equations are equivalent to the flatness condition of the following connection form $\alpha$:
		\begin{align*}
			\alpha
			= \left( w_t + \tfrac{1}{\lambda} W^T \right) dt + \left( -w_{\bar{t}} + \lambda W \right) d\bar{t} 
		\end{align*}
where $\lambda \in \CC^*$ is a complex parameter and
\begin{align}\label{def-of-w-and-W}
	w = \diag (w_0, \cdots, w_n), \hspace{3mm}
	\footnotesize
	W = \left(
		\begin{array}{c|c|c|c}
			& e^{w_1-w_0} & & \\  \hline
			& & \ddots & \\  \hline
			& & & e^{w_n - w_{n-1}} \\ \hline 
			e^{w_0 - w_n} & & & 
		\end{array}
	\right).
\end{align}

By direct computation, one can prove that $\alpha$ satisfies the following conditions.
\begin{itemize}
	\item Cyclic symmetry: $\tau(\alpha(\lambda)) = \alpha(e^{2 \pi \ii / (n+1)} \lambda)$,
	\item Anti-symmetry: $ \sigma(\alpha(\lambda)) = \alpha(-\lambda) $,
	\item $c$-reality: $c(\alpha^{\prime}( \lambda)) = \alpha^{\prime \prime}(1/\bar{\lambda})$ where $\alpha =\alpha^{\prime} dt + \alpha^{\prime \prime} d\bar{t} $.
\end{itemize}

The automorphisms $\tau$, $\sigma$, and $c$ on $\mathfrak{sl}_{n+1}\CC$ are defined by
\begin{equation} \label{automorphisms}
	\tau(X) = d_{n+1}^{-1}Xd_{n+1}, \hspace{3mm} \sigma(X) = - \Delta X^T \Delta,  \hspace{3mm} c(X) = \Delta \bar{X} \Delta,
\end{equation}
where $d_{n+1} = \diag(1, \omega, \cdots, \omega^n)$ ($\omega = e^{\frac{1}{n+1} 2\pi\ii}$) and  $\Delta = (\delta_{i,n-j})_{0 \leq i,j \leq n}$.
Here $X^T$ denotes the transpose of $X \in \mathfrak{sl}_{n+1}\CC$, and $\bar{X}$ is its complex conjugate.
Note that the $c$-reality condition implies that $\alpha |_{|\lambda|=1}$ is an $\mathfrak{sl}_{n+1}^{\Delta}\RR$-valued connection form, where 
$\mathfrak{sl}_{n+1}^{\Delta}\RR = \{ X \in \mathfrak{sl}_{n+1}\CC \,| \, c(X) = X \} \cong \mathfrak{sl}_{n+1}\RR$.

Using the radial condition, one has the radial version of the tt*-Toda equations
\begin{equation*}
	(xw_x)_x = 2x[W^T, W].
\end{equation*}
This is the compatibility condition of the following system
\begin{equation*}
	\left\{
	\begin{aligned}
		&\Psi_{\zeta} = \left( -\tfrac{1}{\zeta^2}  W - \tfrac{1}{\zeta} x w_x + x^2 W^T   \right) \Psi \\
		&\Psi_{x} = \left( \tfrac{1}{ x \zeta} W + x \zeta W^T   \right) \Psi
	\end{aligned}
	\right.
\end{equation*}
where $\zeta = \lambda /t$. 
This is the isomonodromy condition of the associated ODE of the flat sections of the dual connection to the following connection form
\begin{equation*}
			\hat{\alpha} = \left[ -\tfrac{1}{\zeta^2} W^T - \tfrac{1}{\zeta} xw_x + x^2 W \right]	 d\zeta.	
\end{equation*}

By direct computation, one can prove that $\hat{\alpha}$ satisfies the following conditions.
\begin{itemize}
	\item Cyclic symmetry: $\tau(\hat{\alpha}(\zeta)) = \hat{\alpha}(e^{2 \pi \ii / (n+1)} \zeta)$,
	\item Anti-symmetry: $\sigma(\hat{\alpha}(\zeta)) = \hat{\alpha}(-\zeta)$,
	\item $c$-reality: $c(\hat{\alpha}( \zeta)) = \hat{\alpha}(1/ (x^2 \bar{\zeta}))$,
	\item $\theta$-reality: $ \theta(\hat{\alpha}(\zeta)) = \hat{\alpha}(\bar{\zeta})$.
\end{itemize}

The automorphism $\theta$ on $\mathfrak{sl}_{n+1}\CC$ is defined by
\begin{equation*} 
	\theta(X) = \bar{X}.
\end{equation*}

\subsection{Stokes data (odd case)} \label{subsection-Stokes-data}
We review the Stokes data of the tt*-Toda equations when $n+1 = 2d+1$ with $d \in \ZZ_{>0}$, following \cite{GIL4}.
For the even case, we have similar results after suitable modifications (see the remark below).
Regarding complex ODE theory, we refer to \cite{FIKN}.

The ODE for flat sections of $d+ \hat{\alpha}$ is
\begin{equation*}
	\Psi_{\zeta} = \left( -\tfrac{1}{\zeta^2}  W - \tfrac{1}{\zeta} x w_x + x^2 W^T \right)\Psi 
\end{equation*}
which has poles of order $2$ at $\zeta = 0, \infty$.
As in \cite{GIL4} we choose the formal solution at $\zeta = 0$
\begin{equation} \label{def-of-Psitilde}
	\tilde{\Psi}_f^{(0)} (\zeta) = e^{-w} \Omega ( I + \Sigma_{k \geq 1} \psi^{(0)}_k \zeta^k ) e^{\frac{1}{\zeta} d_{n+1}} (d_{n+1})^{-d}
\end{equation}
where $\Omega = (\omega^{ij})_{0 \leq i,j \leq n}$.

As an initial Stokes sector, we take
\begin{equation} \label{initial-Stokes-sector}
	\Omega_1^{(0)} := \{ \zeta \in \widetilde{\CC}^* \,|\, -\tfrac{\pi}{2} - \tfrac{\pi}{2(n+1)} < \arg \zeta < \tfrac{\pi}{2} + \tfrac{\pi}{2(n+1)}  \}.
\end{equation}
$\widetilde{\CC}^*$ is the universal covering of $\CC^*$. Define further Stokes sectors by
\begin{equation} \label{Stokes-sector}
	\Omega^{(0)}_{k + \frac{1}{n+1}} := e^{-\frac{\pi}{n+1}\ii}\Omega_k^{(0)}, \hspace{3mm} k \in \tfrac{1}{n+1}\ZZ.
\end{equation}

Then there exists a unique canonical solution $\widetilde{\Psi}^{(0)}_k$ for each $k \in \tfrac{1}{n+1}\ZZ$, whose asymptotic behaviour on the Stokes sector $\Omega^{(0)}_k$ is determined by the formal solution, 
i.e., $\widetilde{\Psi}^{(0)}_k (\zeta) \sim \tilde{\Psi}_f^{(0)} (\zeta)$ as $\zeta \to 0$ on $\Omega^{(0)}_k$. 
Solutions $\widetilde{\Psi}^{(0)}_k$ extend by analytic continuation to the whole $\widetilde{\CC^*}$.
Stokes matrices and Stokes factors are defined by

	\begin{align}
		& \tilde{\Psi}_{k+1}^{(0)}(\zeta) = \tilde{\Psi}_k^{(0)}(\zeta) \tilde{S}^{(0)}_k \hspace{2mm} \text{on $\Omega^{(0)}_k \cap \Omega^{(0)}_{k+1}$}, \label{Stokes-matrices} \\
		& \tilde{\Psi}_{k+\frac{1}{n+1}}^{(0)}(\zeta) =\tilde{\Psi}_k^{(0)}(\zeta) \tilde{Q}^{(0)}_k \hspace{2mm} \text{on $\Omega^{(0)}_k \cap \Omega^{(0)}_{k+\frac{1}{n+1}}$}. \label{Stokes-factors}
	\end{align}

For $\zeta = \infty$, we introduce $\tilde{\Psi}_k^{(\infty)}(\zeta)$, $\tilde{S}^{(\infty)}_k$ and $\tilde{Q}^{(\infty)}_k$ in the similar way.
The connection matrix is defined by $\tilde{\Psi}_{k}^{(\infty)} = \tilde{\Psi}_{k}^{(0)} E_k$.

The Stokes factors satisfy the following formulae:
\begin{align}
		&\Pi^{-1} \tilQ_{k+\frac{2}{n+1}}\Pi = \tilQ_k \label{cyclicsymmetry-for-Stokesfactors} \\
		& \tilQ_{k+1} = (\tilQ_k)^{-T}. \label{antisymmetry-for-Stokesfactors}
\end{align}
Here  
$\Pi = $
$\left(
\begin{array}{cc}
	0& I_n \\ 
	1 & 0
\end{array}
\right)$ and
$I_n$ is the $n \times n$ identity matrix.
The first and second formula follow from the cyclic symmetry and anti-symmetry conditions. (See (4.3) in \cite{GIL4} and Lemma 3.2 in \cite{GIL3}.)

One can prove that all Stokes matrices and Stokes factors at $\zeta = 0$ and $\zeta = \infty$ are determined by two consecutive Stokes factors $\tilde{Q}^{(0)}_1$ and $\tilde{Q}^{(0)}_{1+\frac{1}{n+1}}$ (see Section 4 in \cite{GIL4}). 

\begin{defi}[\cite{GIL4}] \label{def-of-QQPi}
$\tilde{M}^{(0)} = \tilde{Q}^{(0)}_1\tilde{Q}^{(0)}_{1+\frac{1}{n+1}} \Pi$.
\end{defi}

\begin{theo} [\cite{GIL4}] \label{Thm-of-Stokes-data}
The Stokes data $(s_1, \dots, s_n) $ corresponding to the solution $w$ of the tt*-Toda equations (as in Theorem \ref{asymptotic-data}) are defined as the coefficients of the characteristic polynomial of $\tilde{M}^{(0)}$, i.e.,
\begin{equation*}
	\det( xI_{n+1} - \tilde{M}^{(0)}) = x^{n+1} + \sum_{i=1}^n (-1)^i s_i x^{n+1-i} -1.
\end{equation*}
Explicitly, $s_i$ is the $i$-th symmetric function of 
\begin{equation}\label{asymptotic-stokes}
e^{(2m_0 + n) \frac{\pi \ii}{n+1}}, \,\,\, e^{(2m_1 + n-2) \frac{\pi \ii}{n+1}}, \,\,\,  \dots, \,\,\, e^{(2m_{n-1} -n+2) \frac{\pi \ii}{n+1}}, \,\,\, e^{(2m_n - n) \frac{\pi \ii}{n+1}}.
\end{equation}
\end{theo}

\begin{rem}
When $n+1$ is even, some modifications are needed. 
First, we define $\tilde{\Psi}_f^{(0)}$ by $e^{-w} \Omega ( I + \Sigma_{k \geq 1} \psi^{(0)}_k \zeta^k ) e^{\frac{1}{\zeta} d_{n+1}} d_{n+1}^{\frac{1}{2}} $
The initial Stokes sector is taken as 
\begin{equation*}
		\Omega_1^{(0)} := \{ \zeta \in \widetilde{\CC}^* \,|\, -\tfrac{\pi}{2} - \tfrac{\pi}{n+1} < \arg \zeta < \tfrac{\pi}{2}  \}.
\end{equation*}
A further difference is that  $\tilde{M}^{(0)}$ is defined by  $\tilde{Q}^{(0)}_1\tilde{Q}^{(0)}_{1+\frac{1}{n+1}} \hat{\Pi}$ 
where $\hat{\Pi} = 
\left(
\begin{array}{cc}
	0& I_n \\ 
	-1 & 0
\end{array}
\right)$.
The Stokes data are the coefficients of the characteristic polynomial of $\tilde{M}^{(0)}$ without signs, i.e.,
\begin{equation*}
	\det(xI_{n+1} - \tilde{M}^{(0)}) = x^{n+1} + \sum_{i=1}^n s_i x^{n+1-i} +1.
\end{equation*}
Note that the Stokes data $s_i$ is the $i$-th symmetric function of the complex numbers (\ref{asymptotic-stokes}) whether $n+1$ is odd or even.
\end{rem}

\subsection{Holomorphic data}
We shall use a subsidiary connection form $\omega$ as in \cite{GIL4}.
To reconstruct $\alpha$ from $\omega$, one can use the Iwasawa factorization theorem for loop groups (see \cite{GIL4} and references therein).

Let 
\begin{equation*} 
	\omega  =  \frac{1}{\lambda}
	\begin{pmatrix}
		& & & z^{k_0} \\
		z^{k_1} & & & \\
		& \ddots & &  \\
		& & z^{k_n} &
	\end{pmatrix}dz
\end{equation*}
where $t = \tfrac{n+1}{N}z^{\frac{N}{n+1}}$ and $N =  \Sigma_{i=0}^n (k_i + 1)$.
We call $k_i$ the holomorphic data; we assume that they satisfy $k_i \geq -1$ and $k_i = k_{n+1-i}$.
The connection form $\omega$ satisfies 
\begin{itemize}
	\item Cyclic symmetry: $\tau(\omega (\lambda)) = \omega (e^{2\pi \ii / (n+1)} \lambda)$.
	\item Anti-symmetry: $\sigma(\omega (\lambda)) = \omega (-\lambda)$.
\end{itemize}
The relationship with the asymptotic data $m_i$ is given by
\begin{equation}\label{holomorphic-data-asymptotic-data}
 m_{i-1} - m_i + 1 = \tfrac{n+1}{N}(k_i+1).
\end{equation}
Note that if $m_i$ are fixed, the holomorphic data $k_i+1$ are only determined up to a multiplicative constant. 
However one obtains a one-to-one correspondence between holomorphic data and the asymptotic data for each fixed $N>0$.

As a summary, we have the following theorem.
\begin{theo}[\cite{GIL4}] \label{summaryoftheresult}
For each $N>0$, the global solutions $w$ of the tt*-Toda equations are in one-to-one correspondence with the asymptotic data, the Stokes data, and the holomorphic data. Moreover, the $i$-th Stokes data $s_i$ is the $i$-th symmetric function of the diagonal entries of $e^{(m + \rho) \frac{2 \pi \ii}{n+1}}$ where $\rho = \diag (\frac{n}{2}, \frac{n}{2}-1 , \cdots, -\frac{n}{2})$, and $k_i$ is given by
\begin{equation*}
 m_{i-1} - m_i + 1 = \tfrac{n+1}{N}(k_i+1).
\end{equation*}
\end{theo}

An equivalent result was also obtained by T. Mochizuki in \cite{Mo1}, \cite{Mo2}.

\begin{rem}
As $m \in \mathcal{A}_{n+1}$, the $k_i$ and $s_i$ are restricted by various conditions. In particular the condition $k_i \geq -1$ is equivalent to $m \in \mathcal{A}_{n+1}$.
\end{rem}

\begin{theo} [\cite{GH1}] \label{M-diagonal}
Assume that all $k_i > -1$. Then $\tilde{M}^{(0)}$ is conjugate to the diagonal matrix $e^{\frac{2\pi \ii}{n+1} (m + \rho)}$.
\end{theo}

We shall mention some examples related to geometry.
The following Lie-theoretic observations are due to \cite{GH1}.

Let $\mathfrak{g} = \mathfrak{sl}_{n+1}\CC$. We choose the standard root system as follows:
\begin{itemize}
	\item Cartan subalgebra: $\mathfrak{h} = \{ \diag(h_0, \cdots, h_n) \, | \, h_i \in \CC, \, \sum_{i=0}^n h_i = 0\}$,
	\item roots: $x_i - x_j$ ($0 \leq i \neq j \leq n$) where $x_i : \diag(h_0, \cdots, h_n) \mapsto h_i$, and
	\item simple roots: $\alpha_i = x_{i-1} - x_i$ ($1 \leq i \leq n$). 
\end{itemize}
By using the Killing form $B(X,Y) = \tr (XY)$, we identify $\mathfrak{h}$ and $\mathfrak{h}^*$.
Then we identify the basic weights with $\epsilon_1, \cdots \epsilon_n \in \mathfrak{h}$ where $\alpha_i (\epsilon_j) = \delta_{ij}$.

In this setting, we have $\rho = \epsilon_1 + \cdots + \epsilon_n$, and then
the set of the asymptotic data $\mathcal{A}_{n+1}$ is identified with the fundamental Weyl alcove by 
\begin{equation*}
	\mathcal{A}_{n+1} \ni m \mapsto \tfrac{1}{n+1}( m + \rho) \in \mathfrak{A}_{n+1} = \left\{ \sum_{i=1}^n v_i \epsilon_i \, \middle| \, v_i \geq 0, \,\, \Sigma_{i=1}^n v_i \leq 1 \right\}.
\end{equation*}

\begin{defi} [\cite{HK}]
We call the $1$-parameter family $\lambda \rho \subset \mathcal{A}_{n+1}$ ($\lambda \in [-1, \tfrac{1}{n}]$) the $\rho$-line.
\end{defi}

Several examples of global solutions corresponding to integer Stokes data $s \in \ZZ^n$ (integer Stokes solutions, briefly) are studied in \cite{GL}.
For certain values, such solutions have geometric origins. For example, when the holomorphic data $k_0=0$, $k_1 = \cdots = k_n = -1$, the connection form $\omega$ can be regarded as the Dubrovin connection of the quantum cohomology of the projective space $QH^*(\mathbb{C}P^n)$.
In the case $m=-\rho$, the Dubrovin connection is also known as the Frenkel-Gross connection. See \cite{GH2} and \cite{FG}.

\begin{exa} \label{integer-Stokes-solution-example}
The following are examples of ``geometric solutions" on the $\rho$-line.
\begin{itemize}
	\item $QH^*(\mathbb{C}P^n)$: asymptotic data $m = -\rho$, holomorphic data $k_0=0$, $k_1 = \cdots = k_n = -1$,
		 Stokes data $s_i = \tiny \begin{pmatrix} n+1 \\ i \end{pmatrix}$. 
		 
	\item $A_{n+1}$-singularity: asymptotic data $m = -\tfrac{1}{n+2}\rho$, holomorphic data $k_0=1$, $k_1 = \cdots = k_n = 0$,
		 Stokes data $s_i = 1$. 
		 
	\item $QH^*(\mathbb{P}^{1,n})$: asymptotic data $m = \tfrac{1}{n}\rho$, holomorphic data $k_0=-1$, $k_1 = \cdots = k_n = \frac{1-n}{n}$,
		 Stokes data $s_1 = -1$, $s_2 = \cdots = s_{n-1}=0$, $s_n = -1$.
		 
	\item Trivial solution: all data $=0$. 
\end{itemize}
where $\rho = \diag(\tfrac{n}{2}, \tfrac{n}{2}-1, \cdots , -\tfrac{n}{2})$.
\end{exa}

It is of interest that these ``geometric solutions" all lie on the same line.
Furthermore, a converse is true:

\begin{theo}[\cite{HK}]
Assume $n + 1 > 3$ and $m \in \lambda \rho$.
Then the corresponding Stokes data $s = (s_1, \cdots, s_n) \in \ZZ^n$ if and only if 
$\lambda = -1, -\frac{1}{n+2}, 0,  \frac{1}{n}$. In other words, there are only the four integer Stokes solutions on the $\rho$-line.
They are the solutions given in Example \ref{integer-Stokes-solution-example}.
\end{theo}

\begin{rem}
When $n + 1 = 2$ and $3$, the set of the asymptotic data $\mathcal{A}_{n+1}$ itself can be regarded as the $\rho$-line because we only have $m_0$, and the range of $m_0$ is $[-\frac{1}{2}, \frac{1}{2}]$ and $[-1, \frac{1}{2}]$ respectively.
In these cases it turns out that there are $5$ integer Stokes solutions.

When $n+1=2$, for the five integer values $s = \pm2, \pm 1, 0$ we have only three inequivalent solutions which correspond to $QH^*(\mathbb{C}P^1)$,
the $A_2$-singularity, and the trivial solution, respectively. 
In \cite{HK}, the authors studied the correspondence between the equivalence classes of global solutions with respect to the symmetry and the braid group orbits of the Stokes data.

When $n + 1 = 3$, for the five integer values $s = 3, 2, 1, 0, -1$, there are five solutions corresponding to $QH^*(\mathbb{C}P^2)$, the Hurwitz space $\mathcal{H}_{1,2}$, the $A_3$-singularity, the trivial solution, and $QH^*(\mathbb{P}^{1,2})$, respectively.

Regarding such geometric solutions, one can find further information in \cite{CV2} and \cite{HL}.

\end{rem}

\subsection{Stokes data of the Dubrovin connection and the tt*-Toda equations}
Using the radial condition in the same way that was done for $\alpha$, we obtain another isomonodromy connection $d + \hat{\omega}$
\begin{equation}
	\hat{\omega}(\zeta) = \left[ -\tfrac{1}{\zeta^2} \Pi^T  + \tfrac{1}{\zeta} m  \right] d \zeta
\end{equation}
 with an irregular singular point at $\zeta = 0$ and a regular singular point at $\zeta = \infty$ where $\zeta = \lambda /t$. 
 This connection form satisfies 
 \begin{itemize}
 	\item Cyclic symmetry: $\tau(\hat{\omega}(\zeta)) = \hat{\omega}(e^{2 \pi \ii / (n+1)} \zeta)$,
	\item Anti-symmetry: $\sigma(\hat{\omega}(\zeta)) = \hat{\omega}(-\zeta)$,
	\item $\theta$-reality: $ \theta(\hat{\omega}(\zeta)) = \hat{\omega}(\bar{\zeta})$.
 \end{itemize}
 Let us consider the ODE for flat sections of the dual connection to $d + \hat{\omega}$:
\begin{equation} \label{matrix-ODE-omegahat}
	\chi_{\zeta} = \left( -\frac{1}{\zeta^2} \Pi + \frac{1}{\zeta} m  \right) \chi,
\end{equation}
where $\Pi$ denotes the matrix defined in Theorem \ref{def-of-QQPi}, and $m$ is defined in (\ref{def-of-w-and-W}).
This ODE has an irregular singular point at $\zeta = 0$ and a regular singular point at $\zeta = \infty$.

As in the case of $\tilde{\Psi}_f^{(0)}$ in Section \ref{subsection-Stokes-data}, we can choose the formal solution $\tilde{\chi}^{(0)}_f $
 \begin{equation*}
 	\tilde{\chi}^{(0)}_f = \Omega \left( I + \sum_{k\geq 1} x^{(0)}_k \zeta^k \right) e^{\frac{1}{\zeta} d_{n+1}}(d_{n+1})^{-d},
 \end{equation*}
 where $\Omega$ is as in (\ref{def-of-Psitilde}).
There exists a canonical solution $\tilde{\chi}^{(0)}_k$ which is holomorphic on the Stokes sectors $\Omega^{(0)}_k$ and is asymptotic to $\tilde{\chi}^{(0)}_f$ as $\zeta \to 0$, as in Section \ref{subsection-Stokes-data}.
Then we define the Stokes matrices and the Stokes factors with respect to the solutions $\tilde{\chi}^{(0)}_k$ in the same manner.

\begin{theo}\cite{GIL4} \label{Iwasawa-Stokes}
The Stokes factors of $\hat{\alpha}$ at $\zeta = 0$ are the same as the Stokes factors of $\hat{\omega}$ at $\zeta = 0$.
\end{theo}
By this theorem the Stokes matrices and factors with respect to the solutions $\tilde{\chi}^{(0)}_k$ may also be denoted by $\tilS_k$ and $\tilQ_{k}$.

Stokes matrices and their braid group orbits for the ODE of flat sections of $d + \hat{\omega}$ were studied by Dubrovin in \cite{Du2} and Guzzetti in \cite{Gu} in the case that arose from the quantum cohomology of the projective space $QH^{*}(\mathbb{C}P^n)$. 
As we saw in Example \ref{integer-Stokes-solution-example}, the tt*-Toda equations have a solution which corresponds to $QH^{*}(\mathbb{C}P^n)$. 
In this paper, we study the properties of the Stokes data using both connection forms $\hat{\alpha}$ and $\hat{\omega}$.

\section{Braid group action and distinguished bases}\label{section-braid}

With the motivation of the previous sections, we shall study in the remainder of this article the braid group action in the context of the tt* equations.
As stated in Theorem \ref{Iwasawa-Stokes}, the Stokes data for $\hat{\alpha}$ at $\zeta = 0$ can be calculated using $\hat{\omega}$.
However, they form a proper subset because of the tt* restriction $m \in \mathcal{A}_{n+1}$.
This is one reason why we do not simply quote the results of \cite{Du2} and \cite{Gu}.
The second reason is that we consider all solutions of the tt*-Toda equations, not just the solution that arises from the quantum cohomology of the projective space $QH^*(\mathbb{C}P^n)$.
A third reason is that we wish to study the braid group action from the perspective of distinguished bases following the framework of [HL], which will be reviewed next.

We shall give some basic definitions following \cite{HL}.
Let $H_{\ZZ}$ be a free $\ZZ$-module of rank $n+1$ and $L$ a nondegenerate bilinear form on $H_{\ZZ}$.
\begin{defi}
The pair $(H_{\ZZ},L)$ is called a unimodular bilinear lattice if there exists a $\ZZ$-basis $\underline{e} = (e_1, \cdots, e_{n+1})$ of $H_{\ZZ}$ such that $\det L(\ee, \ee)=1$.
Let $T_{n+1}(\ZZ)$ be the set of all upper triangular matrices with $1$'s on the diagonal, i.e.,
\begin{equation*}
	T_{n+1}(\ZZ) := \{ S \in M_{n+1}(\ZZ)  \,| \, S_{ij}=0 \,\, \text{for $i>j$}, \,\, S_{ii}=1 \}.
\end{equation*}
A basis $\underline{e}$ is called triangular if $L(\underline{e}, \underline{e})^T \in T_{n+1} (\ZZ)$.
The set of triangular bases is denoted by $\mathcal{B}^{tri}$.
\end{defi}

Let $(H_{\ZZ}, L, \ee)$ be a unimodular bilinear lattice with a triangular basis such that $S = L(\ee, \ee)^T$. Then we have the following objects.
\begin{enumerate}[(i)]
	\item A symmetric bilinear form $I^{(0)} : H_{\ZZ} \times H_{\ZZ} \to \ZZ$ and a skew-symmetric bilinear form $I^{(1)} : H_{\ZZ} \times H_{\ZZ} \to \ZZ$ with
		\begin{align*}
			& I^{(0)} = L^T + L, \hspace{3mm} I^{(0)}(\ee, \ee) = S + S^T \\
			& I^{(1)} = L^T - L, \hspace{3mm} I^{(1)}(\ee, \ee) = S - S^T
		\end{align*}
	\item An automorphism $M : H_{\ZZ} \to H_{\ZZ}$ defined by
		\begin{equation} \label{monodromy-HL}
			L(Mv, w) = L(w, v), \hspace{3mm} M(\ee) = \ee(S^{-1}S^T)
		\end{equation}
		where $v, w \in H_{\ZZ}$.
	\item The set of $L_1$-cycles $R^{(0)} := \{ v \in H_{\ZZ} \, | \, L(v,v) =1 \}$ and the set $R^{(1)} = H_{\ZZ}$.
	\item For $p \in \{0,1\}$ and $v \in R^{(p)}$ the reflection (if $p=0$) respectively transvection (if $p=1$) $s_v^{(p)} \in \Aut(H_{\ZZ}, I^{(p)})$ defined by
		\begin{equation*}
			s^{(p)}_v(w) := w - I^{(p)}(v,w)v
		\end{equation*}
		for $w \in H_{\ZZ}$.
		Here, $\Aut(H_{\ZZ}, I^{(p)})$ is a subgroup of $\Aut(H_{\ZZ})$, the automorphism group of $H_{\ZZ}$, consisting of automorphisms that preserve $I^{(p)}$. That is, $A \in \Aut(H_{\ZZ}, I^{(p)})$ if and only if $I^{(p)}( A(w_1), A(w_2)) = I^{(p)}(w_1, w_2)$ for all $w_1, w_2 \in H_{\ZZ}$.
\end{enumerate}

\begin{theo}[\cite{BH}, \cite{HK2}] \label{product-of-reflections-monodromy}
Let $(H_{\ZZ}, L, \ee)$ be a unimodular bilinear lattice with a triangular basis. Let $p \in \{0,1\}$. Then
\begin{equation}
	s_{e_1}^{(p)} \cdots s_{e_{n+1}}^{(p)} = (-1)^{p+1} M.
\end{equation}
\end{theo}

Let $\Br_{n+1}$ be the braid group with $n+1$ strings and $\{ \pm 1 \}^{n+1}$ be the ``sign group" (which is generated by the elements $\delta_j = \text{diag}( (-1)^{\delta_{ij}} )_{i=1, \dots,n+1} \in \{ \pm 1 \}^{n+1}$).
The following relations define an action of $\{ \pm 1 \}^{n+1}$ on $\Br_{n+1}$, and hence a semi-direct product
\begin{align*}
	& \sigma_j \delta_i \sigma_j^{-1} = \delta_i, \hspace{3mm} i \in \{ 1, \cdots, n+1  \} \setminus \{ j, j+1 \},\\
	&  \sigma_j \delta_j \sigma_j^{-1} = \delta_{j+1}, \hspace{3mm}  \sigma_j \delta_{j+1} \sigma_j^{-1} = \delta_{j}.
\end{align*}

\begin{lem}[Lemma 3.10 in \cite{HL}] \label{def-of-braid-action}
Let $(H_{\ZZ},L)$ be a unimodular bilinear lattice of rank $n+1 \geq 2$. Fix $ p \in \{ 0,1 \}$.
The following formulae define an action of $\Br_{n+1} \ltimes \{ \pm 1 \}^{n+1}$ on the set $(R^{(p)})^{n+1}$,
\begin{align*}
	& \sigma_j(\underline{v}) = (v_1, \cdots, v_{j-1}, s_{v_j}^{(p)}(v_{j+1}), v_j, v_{j+2}, \cdots, v_{n+1} ), \\
	& \sigma^{-1}_j(\underline{v}) = (v_1, \cdots, v_{j-1}, v_{j+1}, (s_{v_{j+1}}^{(p)} )^{-1}(v_{j}), v_{j+2}, \cdots, v_{n+1} ), \\
	& \delta_j(\underline{v}) = ( v_1, \cdots, v_{j-1}, -v_j, v_{j+1}, \cdots, v_{n+1} ),
\end{align*}
for $\underline{v} = (v_1, \cdots, v_{n+1}) \in (R^{(p)})^{n+1}$.
\end{lem}

\begin{lem}[Lemma 3.11 and Lemma 3.13 in \cite{HL}]
	\label{braid-group-action}
Let $(H_{\ZZ}, L)$ be a unimodular bilinear lattice with a nonempty set $\mathcal{B}^{tri}$ of triangular bases.
The actions in Lemma \ref{def-of-braid-action} of $\Br_{n+1} \ltimes \{ \pm 1 \}^{n+1}$ on $(R^{(p)})^{n+1}$ both restrict to the same action on $\mathcal{B}^{tri}$. The action can also be written as follows:
\begin{align*}
	& \sigma_j(\underline{v}) = (v_1, \cdots, v_{j-1}, v_{j+1} - L(v_{j+1}, v_j)v_j, v_j, v_{j+2}, \cdots, v_{n+1} ), \\
	& \sigma^{-1}_j(\underline{v}) = (v_1, \cdots, v_{j-1}, v_{j+1}, v_j - L(v_{j+1}, v_j)v_{j+1}, v_{j+2}, \cdots, v_{n+1} ), \\
	& \delta_j(\underline{v}) = ( v_1, \cdots, v_{j-1}, -v_j, v_{j+1}, \cdots, v_{n+1} ),
\end{align*}
for $\underline{v} = (v_1, \cdots, v_{n+1}) \in \mathcal{B}^{tri}$.
There exists a compatible action of $\Br_{n+1} \ltimes \{ \pm 1 \}^{n+1}$ on $T_{n+1}(\ZZ)$ induced by the map 
\begin{equation} \label{projection}
	\mathcal{B}^{tri} \to T_{n+1}(\ZZ), \hspace{2mm} \underline{v} \mapsto L(\underline{v}, \underline{v})^T.
\end{equation}
\end{lem}

\begin{defi}
Let $(H_{\ZZ}, L)$ be a unimodular bilinear lattice with a nonempty set $\mathcal{B}^{tri}$ of triangular bases.
Given a triangular basis $\underline{e} \in \mathcal{B}^{tri}$ and $S = L(\ee, \ee)^T \in T_{n+1}(\ZZ)$, we define
\begin{align*}
	& \text{$\mathcal{B}^{dist} := \Br_{n+1} \ltimes \{ \pm 1 \}^{n+1} (\ee)$, and} \\
	& \text{$\mathcal{S}^{dist} := \Br_{n+1} \ltimes \{ \pm 1 \}^{n+1} (S)$.}
\end{align*}
An element of $\mathcal{B}^{dist}$ is called a distinguished basis and an element of  $\mathcal{S}^{dist}$ a distinguished matrix.
\end{defi}

\begin{defi}[Definition 3.24 in \cite{HL}] \label{the-map-Z}
Let $(H_{\ZZ}, L, \underline{e})$ be a unimodular bilinear lattice of rank $n+1 \geq 2$ with a triangular basis $\underline{e}$.
For $(\alpha, \varepsilon) \in \Br_{n+1} \ltimes \{ \pm 1 \}^{n+1}$ define an automorphism $Z((\alpha, \varepsilon)) \in \Aut(H_{\ZZ})$ as follows:
\begin{align*}
	Z : \Br_{n+1} \ltimes \{ \pm 1 \}^{n+1} &\to \Aut(H_{\ZZ}), \\
	(\alpha, \varepsilon) &\mapsto Z((\alpha, \varepsilon)) = (\underline{e} \to (\alpha, \varepsilon)(\underline{e})).
\end{align*}
\end{defi}

\begin{lem}[Lemma 3.25 (c) in \cite{HL}] \label{lem:anti-homZ}
Let $(\Br_{n+1} \ltimes \{ \pm 1 \}^{n+1} )_S$ be the stabilizer of $S = L(\ee, \ee)^T \in T_{n+1}(\ZZ)$.  The restriction of the map $Z$ to $(\Br_{n+1} \ltimes \{ \pm 1 \}^{n+1} )_S$ is also denoted $Z$,
\begin{align*}
	Z : (\Br_{n+1} \ltimes \{ \pm 1 \}^{n+1})_S \to \Aut(H_{\ZZ},L).
\end{align*}
It is a group antihomomorphism.
\end{lem}

\begin{rem} \label{rmk:lattice-to-solution-space}
In the preceding section, we defined the braid group action on unimodular bilinear lattices. 
In the next section, we consider an analogous structure in the complex vector space $V$ of solutions of a certain scalar ODE 
In this setting, the matrix $S$ representing the bilinear form $L$ with respect to a suitable basis of $V$ is a Stokes matrix of the ODE at an irregular singular point. 
A special feature of this ODE is that $S$ can have integer entries with respect to an appropriate basis. 
The free abelian group generated by this basis therefore naturally defines a unimodular bilinear lattice in $V$.
\end{rem}

\section{Main example}

\subsection{A bilinear form on the solution space} \label{subsec:A bilinear vector space}
In this section we assume $n+1$ is odd and we write $n+1 = 2d+1$.
If $n+1$ is even a similar analysis applies but we focus on the odd case, as the case $n+1=3$ will be our main example.
We now carry out the construction anticipated in Remark \ref{rmk:lattice-to-solution-space}. 
We consider the scalar ODE associated with the system \eqref{matrix-ODE-omegahat} and construct a nondegenerate bilinear form on its solution space.

We shall find a basis $X^{(0)}, \cdots, X^{(n)}$ of solutions of the system
\begin{equation*}
	X_{\zeta} = \left( -\frac{1}{\zeta^2} \Pi + \frac{1}{\zeta} m  \right) X, \quad X = (x_0, \cdots, x_n)^T,
\end{equation*}
which is the system in (\ref{matrix-ODE-omegahat}).
Then the first entry $x_0$ satisfies the following scalar ODE
 \begin{equation} \label{scalar-ODE-omegahat}
 	\zeta( \partial_{\zeta} -m_{n} ) \cdots \zeta( \partial_{\zeta} -m_{0}) x_0 = -x_0
 \end{equation}
 where $\partial_{\zeta} = \zeta \tfrac{d}{d\zeta}$.
 
\begin{defi}
Let $V$ denote the vector space consisting of (multi-valued) solutions of the scalar ODE (\ref{scalar-ODE-omegahat}).
\end{defi}

\begin{prop} \label{prop:rotation}
The scalar ODE (\ref{scalar-ODE-omegahat}) is invariant under the rotation
\begin{equation*}
	\zeta \mapsto \omega \zeta, \hspace{2mm} \omega = e^{\frac{1}{n+1}2 \pi \ii}.
\end{equation*}
Consequently, if $x_0(\zeta)$ is a solution of \eqref{scalar-ODE-omegahat} then $x_0(\omega \zeta)$ is also a solution.
 \end{prop}

Following \cite{GIL4}, we introduce a solution $g$ of the scalar ODE \eqref{scalar-ODE-omegahat} to construct a basis of $V$.
Related constructions were used in \cite{Du2} and \cite{Gu} to compute Stokes data.
 
 \begin{lem}[cf. Lemma 6.14 \cite{GIL4}] \label{prop-of-g}
 Let $c$ be a real number such that $\tfrac{-m_i + i}{n+1} > c$ for all $i$. Let $g(\zeta) = $
 \begin{equation} \label{function-g}
 	\int^{c + \ii \infty}_{c - \ii \infty} 
	\prod_{i=0}^n \Gamma \left( \tfrac{-m_i +i}{n+1} -t \right) (n+1)^{-(n+1)t} e^{-\ii (n+1)\pi t} \zeta^{-(n+1)t}dt.
 \end{equation}
 This defines a holomorphic function of $\zeta$ on the sector $-\tfrac{3\pi}{2} < \arg \zeta < -\tfrac{\pi}{2}$. Furthermore
 \begin{enumerate}[(1)]
 	\item $g$ satisfies the scalar ODE (\ref{scalar-ODE-omegahat}).
	\item $g(\zeta) \sim \ii (2 \pi)^{\frac{1}{2}(n+2)} (n+1)^{-\frac{1}{2}} e^{\frac{1}{\zeta}} $ as $\zeta \to 0$ in the sector $-\tfrac{3\pi}{2} < \arg \zeta < -\tfrac{\pi}{2}$. \label{behaviour-at-infty}
 \end{enumerate}
 \end{lem}

Let $C_{asymp} :=\ii (2 \pi)^{\frac{1}{2}(n+2)} (n+1)^{-\frac{1}{2}} $ be the coefficient of the asymptotic formula.
Let $g^{[1]}(\zeta) = \zeta(\partial_{\zeta} - m_0)g(\zeta)$ and $g^{[i+1]} = \zeta(\partial_{\zeta} -m_i)g^{[i]}$, $i=1, \cdots, n$.
\begin{defi} [\cite{GIL4}] \label{def:chig}
Set
\begin{equation}
 	\chi_g := 
	\begin{pmatrix}
		g(\zeta) & \cdots & g(\omega^{-d} \zeta) & \cdots & g(\omega^{-n} \zeta) \\
		-g^{[1]}(\zeta) & \cdots & -\omega^{d}g^{[1]}(\omega^{-d} \zeta) & \cdots & -\omega^{n}g^{[1]}(\omega^{-n} \zeta) \\
		g^{[2]}(\zeta) & \cdots & \omega^{2d}g^{[2]}(\omega^{-d} \zeta) & \cdots & \omega^{2n}g^{[2]}(\omega^{-n} \zeta) \\
			\vdots & & \vdots & & \vdots \\
	\end{pmatrix}.
\end{equation}
Define the normalized fundamental solution by
\begin{equation*}
\hat{\chi}_g := (C_{asymp})^{-1}\chi_g(d_{n+1})^{-d}.
\end{equation*}
For a fundamental solution $\chi$, write $\chi^{row}$ for its first row.
Define
\begin{equation*}
g_0(\zeta) := (C_{asymp})^{-1}g(\zeta), \quad \text{and} \quad
g_\ell(\zeta)
:= \omega^{\ell d}g_0(\omega^\ell\zeta),
\quad \ell\in\ZZ.
\end{equation*}
Then, $ \hat{\chi}_g^{row} = (g_0,g_{-1},\ldots,g_{-n})$.
\end{defi}

Then we define the following associated linear map on the vector space $V$.
\begin{defi} \label{def-of-phi}
$\phi(f)(\zeta) = \omega^{d}f(\omega \zeta)$.
Note that $\phi(g_{\ell})(\zeta) = g_{\ell+1}(\zeta)$ and therefore, $g_0$ is a cyclic vector for $\phi$.
\end{defi}

By Proposition \ref{prop:rotation}, the entries of $\chi_g^{row}$ are solutions of the scalar ODE (\ref{scalar-ODE-omegahat}).
It follows from (\ref{behaviour-at-infty}) in Lemma \ref{prop-of-g} that $\chi_g^{row}$ is a basis of $V$.

The following matrix $F$ plays an important role in our calculations.

\begin{defi}[\cite{Ho}]
When $n+1 = 2d+1$, the matrix $F$ is defined as follows:
\begin{equation}
	F := 
	\left(
	\begin{array}{c|c}
		L^{[d+1]} & \bm{0} \\ \hline
		\bm{0} & U^{[d]}
	\end{array}
	\right).
\end{equation}
Here, $L^{[m]}$ and $U^{[m]}$ are matrices of size $m$ defined as follows: 
\begin{equation*}
	\begin{aligned}
		&L^{[m]} := I_{m} - \sum_{k=0}^{\ell_1 -1}  p_{m -2k -1} \sum_{j=0}^k E_{m - 2k+j, j+1} - \sum_{k=0}^{\ell_2 -1} p_{m -2k -2} \sum_{j=0}^k E_{m - 2k-1, j+1}, \\
		&U^{[m]} := 
		\left(
		\begin{array}{c|c}
			1 & \bm{0} \\ \hline
			\bm{0} & \Delta L^{[m-1]} \Delta
		\end{array}
		\right).
		\end{aligned}
\end{equation*}
Here, $E_{a,b}$ denotes the $m \times m$ matrix unit whose $(a,b)$-entry is $1$ and whose other entries are $0$. 
In the definition of $U^{[m]}$, $\Delta$ denotes the matrix introduced in (\ref{automorphisms}), taken here to be of size $m-1$.
Moreover, $\ell_1 = [\tfrac{m}{2}]$, $\ell_2 = [\tfrac{m-1}{2}]$, and $p_m = (-1)^{n+1-m} s_{n+1-m}$.
Note that $L^{[1]}$ and $U^{[1]}$ are defined as the $(1 \times 1)$ identity matrix.
\end{defi}

\begin{prop}[Proposition 6.16 in \cite{GIL4}] 
The cyclic symmetry of the ODE (\ref{matrix-ODE-omegahat})
\begin{equation}\label{cyclic-symmetry-of-omegahat}
	\omega^{-d} d_{n+1} \tilde{\chi}^{(0)}_k(\omega^{-1}\zeta) \Pi = \tilde{\chi}_{k-\frac{2}{n+1}}^{(0)} (\zeta)
\end{equation}
implies that the canonical solution $\tilde{\chi}_1^{(0)}$ is related to $\hat{\chi}_g$ by the following formula.
\begin{equation} \label{chione-to-chig}
\tilde{\chi}_1^{(0)} (\zeta) = \hat{\chi}_g(\zeta) F.
\end{equation}
\end{prop}

 \begin{prop}[Proposition 3.3 and Proposition 3.4 in \cite{Ho}] \leavevmode

 \begin{align} \label{canonical-M}
 	& F  \tilde{M}^{(0)} F^{-1} =
\left(
	\begin{array}{c|ccccc}
		s_1 &	 	1 & & & & \\
		-s_2 &	 	 & \ddots  & &\bm{0} & \\
		\vdots &	 	 &  & \ddots &  &   \\ 
		s_{n-1} &	 & \bm{0} & & 1 &  \\ 
		-s_n & & & & & 1 \\ \hline
		1 & 0 & \cdots & \cdots & 0 & 0
	\end{array}
	\right) =: M_F. \\
 \label{canonical-Stokes}
	&F \tilS_1 F^T = 
	\left(
	\begin{array}{ccccccc}
		1 &  -s_n &  s_{n-1} & \cdots & \cdots & -s_2 &  s_1\\
		0& 1  &  -s_n & s_{n-1} &\dots & \dots  & -s_2 \\
		\vdots& \ddots& \ddots  & \ddots &  \ddots &   & \vdots \\
		\vdots& &\ddots & \ddots  & \ddots & \ddots & \vdots \\
		\vdots&  \bm{0}  & &\ddots  & \ddots  & \ddots & s_{n-1}\\
		\vdots& & &  &\ddots  & \ddots & -s_{n}\\
		0 & \dots & \dots & \dots & \dots & 0 & 1
	\end{array}
	\right) =: S_F.
\end{align}
The diagonal entries of $S_F$ are $1$ and $(i,j)$-th entry of $S_F$ is $(-1)^{j-i} s_{n+1-(j-i)}$ when $0 \leq i < j \leq n$. The other entries are $0$.
 \end{prop}

\begin{prop}
The solution $g(\zeta)$ of (\ref{scalar-ODE-omegahat}) satisfies
\begin{equation*}
	g(\zeta) = \sum_{\ell =0}^{n} (-1)^{\ell} s_{\ell + 1} \omega^{-(\ell+1) d}g(\omega^{-(\ell+1)}\zeta).
\end{equation*}
where $s_{n+1} = 1$.
In particular, the functions $g_k$ satisfy the recurrence relation
\begin{equation}  \label{monodromy-of-g} 
	g_k(\zeta)= \sum_{\ell=0}^n (-1)^{\ell}s_{\ell +1}g_{k-(\ell+1)} (\zeta).
\end{equation}
\end{prop}

\begin{proof}
The cyclic symmetry formula \eqref{cyclic-symmetry-of-omegahat} gives
\begin{equation*}
    \omega^{-d} d_{n+1} \hat{\chi}_g(\omega^{-1}\zeta) M_F
    = \hat{\chi}_g(\zeta).
\end{equation*}
Taking the first row, we obtain
\begin{equation*}
    \omega^{-d} \hat{\chi}^{row}_g(\omega^{-1}\zeta) M_F
    = \hat{\chi}^{row}_g(\zeta).
\end{equation*}
Comparing the first components and using the definition of $g_\ell$, we obtain
\begin{equation*}
    g_0(\zeta)
    = \sum_{\ell=0}^n (-1)^\ell s_{\ell+1}
      g_{-(\ell+1)}(\zeta).
\end{equation*}
The first identity follows from the definition of $g_\ell$.
Applying $\phi^k$ and using $\phi^k(g_j)=g_{j+k}$ gives the recurrence
relation for every $k\in\mathbb{Z}$.
\end{proof}
 
By (\ref{canonical-Stokes}) we can define a nondegenerate bilinear form on $V$ for which $\hat{\chi}_g^{row}$ is a triangular basis, as follows.

\begin{defi} \label{the-bilinear-form}
Define the bilinear form $L$ on $V$ by $L(\hat{\chi}_g^{row},\hat{\chi}_g^{row}) = (S_F)^{-T}$.
Since $(S_F)^{-T}$ is unitriangular, $L$ is nondegenerate.
\end{defi}

\begin{prop} \label{phi-preserves-L}
The bilinear form $L$ is preserved by $\phi$, i.e.,
\begin{equation*}
	L(f_1(\zeta), f_2(\zeta)) = L(\phi(f_1(\zeta)), \phi(f_2(\zeta)) )
\end{equation*}
for any $f_1(\zeta)$, $f_2(\zeta) \in V$.
\end{prop}

\begin{proof}
Since $\hat{\chi}_g^{row}$ is a basis of $V$, it suffices to show that 
\begin{equation*}
	L(\hat{\chi}_g^{row}(\zeta), \hat{\chi}_g^{row} (\zeta)) = L (\phi( \hat{\chi}_g^{row}(\zeta) ), \phi( \hat{\chi}_g^{row}(\zeta)) ).
\end{equation*}
By the definition of $L$ and the linear map $\phi$, this formula is equivalent to $(S_F)^{-T} = M_F^T (S_F)^{-T} M_F$ and $(\tilS_1)^{-T} = (\tilde{M}^{(0)})^T ( \tilS_1)^{-T} \tilM$.
The Stokes matrix $\tilS_1$ can be written in terms of the Stokes factors:
\begin{align*}
	\tilS_1 &=  \left( \tilQ_{1} \tilQ_{1\frac{1}{n+1}} \Pi \right)^{\frac{n-2}{2}} \Pi^{-d} \tilQ_{1\frac{n}{n+1}} 
		 =  \left( \tilQ_{1} \tilQ_{1\frac{1}{n+1}} \Pi \right)^{\frac{n-2}{2}}  \tilQ_{1} \Pi^{-d}.
\end{align*}
Therefore, we have
\begin{align*}
	& (\tilde{M}^{(0)})^T ( \tilS_1)^{-T} \tilM \\
	&= (\tilQ_{1} \tilQ_{1\frac{1}{n+1}} \Pi)^T  \left\{   \left( \tilQ_{1} \tilQ_{1\frac{1}{n+1}} \Pi \right)^{\frac{n-2}{2}}  \tilQ_{1} \Pi^{-d}\right\}^{-T}   (\tilQ_{1} \tilQ_{1\frac{1}{n+1}} \Pi) \\
	& =  \left\{ \left( \tilQ_{1} \tilQ_{1\frac{1}{n+1}} \tilde{\Pi} \right)^{\frac{n-2}{2}-1} \tilQ_1 \right\}^{-T}   ( \tilQ_{2\frac{1}{n+1}} \Pi) \tilQ_2  \Pi^{d+1}  \\
	& =  \left\{ \left( \tilQ_{1} \tilQ_{1\frac{1}{n+1}} \tilde{\Pi} \right)^{\frac{n-2}{2}-1}  (\tilQ_1 \tilQ_{1\frac{1}{n+1}} \Pi) \tilQ_1  \Pi^{-d}\right\}^{-T}   \\
	& = (\tilS_1)^{-T}.
\end{align*}
Note that $\Pi^{-d} = \Pi^{d+1}$ and $\Pi^T = \Pi^{-1}$.
In this computation, we use \eqref{cyclicsymmetry-for-Stokesfactors} and \eqref{antisymmetry-for-Stokesfactors}.
\end{proof}

\subsection{Properties of $L$: Characterization} \label{subsec:Characterization}
In this section, we explain why the above definition of the bilinear form $L$ is natural.

Let $(\alpha_0, \alpha_1, \dots, \alpha_n)$ be the diagonal entries of $d_{n+1}$, that is, $\alpha_i = \omega^{i}$ ($i = 0, \dots, n$).
On $\Omega_k^{(0)} \cap \Omega_{k+1}^{(0)}$ ($k \in \tfrac{1}{n+1}\ZZ$), there exists a unique permutation $\sigma_k$ such that
\begin{equation*}
	\Real(\alpha_{\sigma_k(0)}\zeta^{-1}) < \Real(\alpha_{\sigma_k(1)}\zeta^{-1}) < \cdots < \Real(\alpha_{\sigma_k(n)}\zeta^{-1}).
\end{equation*}
Then the associated Stokes matrix $\uppS_k$ (explained below) is upper-triangular with respect to the ordering $\sigma_k$ (see \cite{FIKN}, Chapter $1$, Section $1.4$).
We call this order the dominance order on $\Omega_k^{(0)} \cap \Omega_{k+1}^{(0)}$.
The corresponding permutation matrix is denoted by $\Sigma_k = (\delta_{\sigma_k(i),j})_{0 \leq i,j \leq n}$.
We shall give a concrete example in the case $n=2$ in the next section.

\begin{rem}
We define the dominance order to be the reverse of that in \cite{FIKN} so that the associated Stokes matrices become upper triangular. Our definition of the order agrees with the lexicographical order in \cite{Gu} and \cite{CDG3}. 
\end{rem}

\begin{defi}
For each $k \in \tfrac{1}{n+1}\ZZ$, let $\hat{\chi}_k$ be the canonical solution on $\Omega_k^{(0)}$ associated with the formal solution $\tilde{\chi}^{(0)}_f\Sigma_k$, characterized by the asymptotic behaviour
\begin{equation}
    \hat{\chi}_k(\zeta)
    \sim
    \tilde{\chi}^{(0)}_f(\zeta)\Sigma_k,
    \quad
    \text{as} \quad\zeta \to 0,\ \zeta \in \Omega_k^{(0)}.
\end{equation}
By the uniqueness of the canonical solution, we have $\hat{\chi}_k = \tilde{\chi}^{(0)}_k \Sigma_k$.
\end{defi}

\begin{defi}
We define the matrices $B_k$ and $\uppS_k$ as follows:
\begin{align}
	&\hat{\chi}_{k+\frac{1}{n+1}} = \hat{\chi}_k B_k, \\
	&\tilde{\chi}^{(0)}_{k+1} \Sigma_k = \tilde{\chi}^{(0)}_k \Sigma_k \uppS_k.
\end{align}
\end{defi}

\begin{prop} \label{modified-Stokes} \leavevmode
	\begin{enumerate}
		\item $B_k = \Sigma_k^{-1} \tilQ_k \Sigma_{k+\frac{1}{n+1}}$.
		\item $\uppS_k = \Sigma_k^{-1} \tilS_k \Sigma_k$, and $\uppS_k$ is upper-triangular.
		\item $\uppS_{k+\frac{1}{n+1}} = B_k^{-1} \uppS_k B_k^{-T}$. \label{basis-change-of-bilinear}
		\item The matrices $\Sigma_k$, $B_k$ and $\uppS_k$ are periodic in $k$, namely
			$\Sigma_{k+2} = \Sigma_k$, $B_{k+2} = B_k$ and $\uppS_{k+2} = \uppS_k$. \label{Bk-periodic}
	\end{enumerate}
\end{prop}

\begin{proof}
The proof follows directly from the definitions of $B_k$ and $\uppS_k$, and that of the dominance order. 
We also use the anti-symmetry condition for Stokes factors in (\ref{antisymmetry-for-Stokesfactors}) and $(\Sigma_k)^T = (\Sigma_k)^{-1}$ to prove (\ref{basis-change-of-bilinear}).
\end{proof}

\begin{theo}
$L( \hat{\chi}^{row}_k, \hat{\chi}^{row}_k ) = (\tilde{S}^{\Delta}_k)^{-T}$.
In particular, $\hat{\chi}^{row}_k$ is a triangular basis of $(V, L)$.
\end{theo}

\begin{proof}
By the definition of $L$, we have $L(\hat{\chi}_g^{row}, \hat{\chi}_g^{row}) = (S_F)^{-T}$.
Moreover \eqref{chione-to-chig} gives $\hat{\chi}_1 =  \hat{\chi}_g (F \Sigma_1) $.
Then, it follows that
\begin{equation*}
\begin{aligned}
	&L(\hat{\chi}^{row}_1, \hat{\chi}^{row}_1) = L(  \hat{\chi}^{row}_g (F \Sigma_1) ,  \hat{\chi}^{row}_g (F \Sigma_1) ) =(F \Sigma_1)^T L(\hat{\chi}^{row}_g, \hat{\chi}^{row}_g) (F \Sigma_1)\\
	 &= (F \Sigma_1)^T (S_F)^{-T}(F \Sigma_1).
\end{aligned}
\end{equation*}
By the definition of $S_F$ in \eqref{canonical-Stokes}, $(\tilS_1)^{-T} = (F^{-1} S_F F^{-T})^{-T}$.
Therefore, 
\begin{equation*}
	(F \Sigma_1)^T (S_F)^{-T}(F \Sigma_1) = (\Sigma_1^{-1} \tilS_1 \Sigma_1)^{-T} = (\uppS_1)^{-T}.
\end{equation*}

Suppose that $L( \hat{\chi}^{row}_k, \hat{\chi}^{row}_k ) = (\tilde{S}^{\Delta}_k)^{-T}$.
Then, by (\ref{basis-change-of-bilinear}) in Proposition \ref{modified-Stokes},
\begin{equation*}
	\begin{aligned}
	&L(\hat{\chi}^{row}_{k+\frac{1}{n+1}}, \hat{\chi}^{row}_{k+\frac{1}{n+1}}) = L(\hat{\chi}^{row}_{k}B_k, \hat{\chi}^{row}_{k}B_k) =B_k^T L(\hat{\chi}^{row}_{k}, \hat{\chi}^{row}_{k}) B_k \\
	&= B_k^T (\uppS_{k})^{-T} B_k= (\tilde{S}^{\Delta}_{k+\frac{1}{n+1}})^{-T}.
	\end{aligned}
\end{equation*}
The result follows by induction.
\end{proof}

Thus the canonical solutions of the system \eqref{matrix-ODE-omegahat}, ordered by dominance, give triangular bases of $(V,L)$ through their first rows. In the case $n=2$, we shall further show in Proposition \ref{prop:distinguished-basis-3times3}  that these are distinguished bases with respect to the initial triangular basis $\hat{\chi}^{row}_{1\frac{1}{3}}$. This gives an explicit relation between the canonical solutions and the braid group action in this case, as illustrated in Table \ref{table-of-matrices}. 
The changes of basis described by the matrices $B_k$ in Table \ref{table-of-matrices}
are given alternately by the actions of the braid group generators
$\sigma_1$ and $\sigma_2$.
This correspondence is made explicit in Proposition~4.21.


\begin{table}[htbp]

\renewcommand{\arraystretch}{1.2}
\caption{$\Sigma_k$, $B_k$ and $\tilde{S}^{\Delta}_k$ for $n=2$.}
\label{table-of-matrices}
\begin{tabular}{c|c|c|c|c} \hline
	$k$ & dominance order on $\Omega_k^{(0)} \cap \Omega_{k+1}^{(0)}$ & $\Sigma_k$ & $B_k$ & $\tilde{S}^{\Delta}_k = L(\hat{\chi}^{row}_k, \hat{\chi}^{row}_k)^{-T}$ \\ \hline \hline
	$1$ & $\Real( \omega\zeta^{-1}) < \Real( \zeta^{-1}) < \Real(\omega^2 \zeta^{-1})$ &
	$\begin{pmatrix}
		0 & 1 & 0 \\
		1 & 0 & 0 \\
		0 & 0 & 1
	\end{pmatrix}$ &
	$\begin{pmatrix}
			s & 1 & 0 \\
			1 & 0 & 0 \\
			0 & 0 & 1
		\end{pmatrix}$ &
	$\begin{pmatrix}
			1 & s & s^2-s \\
			0 & 1 & s \\
			0 & 0 & 1
	\end{pmatrix}$ \\ \hline 
	$1\frac{1}{3}$ & $\Real( \zeta^{-1}) < \Real( \omega \zeta^{-1}) < \Real( \omega^2 \zeta^{-1})$ &
	$	\begin{pmatrix}
			1 & 0 & 0 \\
			0 & 1 & 0 \\
			0 & 0 & 1
		\end{pmatrix}$ &
		$\begin{pmatrix}
			1 & 0 & 0 \\
			0 & -s & 1 \\
			0 & 1 & 0
		\end{pmatrix}$ &
	$\begin{pmatrix}
			1 & -s & s \\
			0 & 1 & -s \\
			0 & 0 & 1
		\end{pmatrix}$ \\ \hline
		$1\frac{2}{3}$ & $\Real( \zeta^{-1}) < \Real( \omega^2 \zeta^{-1}) < \Real( \omega\zeta^{-1})$ &
	$\begin{pmatrix}
			1 & 0  & 0 \\
			0 & 0 & 1 \\
			0 & 1 & 0
		\end{pmatrix}$ &
	$\begin{pmatrix}
			s & 1 & 0 \\
			1 & 0 & 0 \\
			0 & 0 & 1
		\end{pmatrix}$ &
	$\begin{pmatrix}
			1 & s & s^2-s \\
			0 & 1 & s \\
			0 & 0 & 1
	\end{pmatrix}$ \\ \hline
		$2$ & $\Real( \omega^2 \zeta^{-1}) < \Real( \zeta^{-1}) < \Real(\omega \zeta^{-1})$ &
	$	\begin{pmatrix}
			0 & 1 & 0  \\
			0 & 0 & 1 \\
			1 & 0 & 0
		\end{pmatrix}$ &
		$\begin{pmatrix}
			1 & 0 & 0 \\
			0 & -s & 1 \\
			0 & 1 & 0
		\end{pmatrix}$ &
	$\begin{pmatrix}
			1 & -s & s \\
			0 & 1 & -s \\
			0 & 0 & 1
		\end{pmatrix}$ \\ \hline
		$2\frac{1}{3}$ & $\Real( \omega^2 \zeta^{-1}) < \Real( \omega \zeta^{-1}) < \Real( \zeta^{-1})$ &
	$\begin{pmatrix}
			0 & 0 & 1 \\
			0 & 1 & 0 \\
			1 & 0 & 0
		\end{pmatrix}$ &
	$\begin{pmatrix}
			s & 1 & 0 \\
			1 & 0 & 0 \\
			0 & 0 & 1
		\end{pmatrix}$ &
	$\begin{pmatrix}
			1 & s & s^2-s \\
			0 & 1 & s \\
			0 & 0 & 1
	\end{pmatrix}$ \\ \hline
			$2\frac{2}{3}$ & $\Real( \omega \zeta^{-1}) < \Real(\omega^2 \zeta^{-1}) < \Real( \zeta^{-1})$ &
	$	\begin{pmatrix}
			0 & 0 & 1 \\
			1 & 0 & 0 \\
			0 & 1 & 0
		\end{pmatrix}$ &
		$\begin{pmatrix}
			1 & 0 & 0 \\
			0 & -s & 1 \\
			0 & 1 & 0
		\end{pmatrix}$ &
	$\begin{pmatrix}
			1 & -s & s \\
			0 & 1 & -s \\
			0 & 0 & 1
		\end{pmatrix}$ \\ \hline
\end{tabular}

\end{table}


\begin{theo}
The bilinear form $L$ on $V$ is uniquely characterized
by the following conditions:
\begin{equation}\label{prop:conditions-for-L}
\left\{
\begin{aligned}
&L(g_{-i},g_{-j})=0,
&&0\le i<j\le n,\\
&L(g_{-i},g_{-i})=1,
&&0\le i\le n,\\
&L(\phi(v),\phi(w))=L(v,w),
&&v,w\in V.
\end{aligned}
\right.
\end{equation}
Here,
$\hat{\chi}_g^{row}=(g_0,g_{-1},\ldots,g_{-n})$.
In particular,
\begin{equation*}
L(\hat{\chi}_k^{row},
  \hat{\chi}_k^{row})
=(\tilde{S}_k^\Delta)^{-T}
\end{equation*}
for every $k\in\tfrac{1}{n+1}\ZZ$.
\end{theo}

\begin{proof}
The form $L$ satisfies these conditions by its definition
and $\phi$-invariance, and the formula for the canonical
bases was established above.

For uniqueness, let $\tilde{L}$ be another bilinear form
satisfying these conditions.
Since
\begin{equation*}
\tilde{L}(g_0,g_{-j})=L(g_0,g_{-j})=\delta_{0j},
\qquad 0\le j\le n,
\end{equation*}
and $(g_0,g_{-1},\ldots,g_{-n})$ is a basis of $V$,
we have $\tilde{L}(g_0,v)=L(g_0,v)$ for every $v\in V$.
Using $\phi^i(g_{-i})=g_0$ and $\phi$-invariance, we obtain
\begin{equation*}
\tilde{L}(g_{-i},v)
=\tilde{L}(g_0,\phi^i(v))
=L(g_0,\phi^i(v))
=L(g_{-i},v)
\end{equation*}
for every $0\le i\le n$ and $v\in V$.
Thus $\tilde{L}=L$.
\end{proof}

\subsection{Properties of $L$: Periodicity} \label{subsec:Periodicity}
In this section, we discuss the relation between the automorphism $\phi$ and the braid group action.

 \begin{prop} \label{braid-and-phi}
$Z(\sigma_1 \sigma_2 \cdots \sigma_n) = \phi \in \Aut(V, L)$. That is,
\begin{equation*}
(\sigma_1 \sigma_2 \cdots \sigma_n) \, \hat{\chi}_g^{row} =  \phi(\hat{\chi}_g^{row}).
\end{equation*}
\end{prop}

\begin{proof}
By the definition of the braid group action in Lemma \ref{braid-group-action}, we have
\begin{equation*}
	\begin{aligned}
		&(\sigma_1 \sigma_2 \cdots \sigma_n )(g_0, \cdots, g_{-n+1}, g_{-n}) = (( s_{g_0}^{(p)} \circ \cdots \circ  s_{g_{-n+1}}^{(p)} ) (g_{-n}), g_0, \cdots, g_{-n+1} )   \\
		= & (( s_{g_0}^{(p)} \circ \cdots \circ  s_{g_{-n+1}}^{(p)} \circ s_{g_{-n}}^{(p)} ) ( (-1)^{p+1}g_{-n}), g_0, \cdots, g_{-n+1} ).
	\end{aligned}
\end{equation*}
Here we use $s_{g_{-n}}^{(p)}(g_{-n}) = (-1)^{p+1} g_{-n}$. 
By Theorem \ref{product-of-reflections-monodromy}, 
\begin{equation*}
( s_{g_0}^{(p)} \circ \cdots \circ  s_{g_{-n+1}}^{(p)} \circ s_{g_{-n}}^{(p)} ) ( (-1)^{p+1}g_{-n})= M(g_{-n})
\end{equation*}
where $M$ is the automorphism defined in \eqref{monodromy-HL}. 
By the definition of $M \in \Aut(V)$, we have 
\begin{align*}
	&M(\hat{\chi}_g^{row}) = M(g_0, \cdots, g_{-n}) = (g_0, \cdots, g_{-n}) S_F (S_F)^{-T} \\
	&= (*, \cdots, *, s_1 g_0 - s_2 g_{-1} + \cdots + s_{n-1} g_{-n+2} - s_n g_{-n+1} + g_{-n} )
\end{align*}
where $*$ is a linear combination of $g_0, \cdots, g_{-n}$.
In particular, $M(g_{-n}) = s_1 g_0 - s_2 g_{-1} + \cdots + s_{n-1} g_{-n+2} - s_n g_{-n+1} + g_{-n} $. 
Therefore, we obtain $(\sigma_1 \sigma_2 \cdots \sigma_n) \, \hat{\chi}_g^{row} =\hat{\chi}_g^{row} M_F$.
On the other hand, applying $\phi$ to $g_0$ and using \eqref{monodromy-of-g}, we obtain $\hat{\chi}_g^{row} M_F = \phi (\hat{\chi}_g^{row} )$.
Finally, using $\phi$-invariance of $L$, we obtain $Z(\sigma_1 \sigma_2 \cdots \sigma_n) \in \Aut(V,L)$.
\end{proof}

\begin{theo} \label{periodicity-of-the-braid}
Consider the solution of the tt$^*$-Toda equations corresponding to holomorphic data $k_0,\dots,k_n \in \mathbb{Z}_{\ge 0}$.
Then,
\begin{equation*}
(\sigma_1 \sigma_2 \cdots \sigma_n )^{\tilde{N}}  \hat{\chi}_g^{row} \, = \hat{\chi}_g^{row},
\end{equation*}
where $\tilde{N} = \frac{N}{\gcd(\alpha_0, \alpha_1, \cdots, \alpha_d)}$, and $\alpha_i = k_i +1$ for $i=0,1,\cdots,n$.
In particular, $\ord(M_F) = \tilde{N}$.
\end{theo}

\begin{proof}
Recall that $Z(\sigma_1\sigma_2\cdots\sigma_n)\in\Aut(V,L)$.
By Lemma \ref{lem:anti-homZ}, we have
\begin{equation*}
(\sigma_1\sigma_2\cdots\sigma_n)^m\hat{\chi}_g^{row}
=
Z((\sigma_1\sigma_2\cdots\sigma_n)^m)\hat{\chi}_g^{row}
=
Z(\sigma_1\sigma_2\cdots\sigma_n)^m\hat{\chi}_g^{row}
=
\hat{\chi}_g^{row}(M_F)^m
\end{equation*}
for $m\in\ZZ$.
It suffices to prove that $\operatorname{ord}(M_F)=\tilde{N}$.

Let $\lambda_i=e^{\frac{\pi\ii}{n+1}(2m_i+n-2i)}$ and $\zeta_N=e^{\frac{2\pi\ii}{N}}$.
Then
$\lambda_{i-1}/\lambda_i
=\exp\bigl(\frac{2\pi\ii}{n+1}(m_{i-1}-m_i+1)\bigr)
=\zeta_N^{\alpha_i}$
by \eqref{holomorphic-data-asymptotic-data}.
By \eqref{holomorphic-data-asymptotic-data} again, we have
\begin{equation*}
\lambda_i=
\begin{cases}
e^{\frac{2\pi\ii}{N}\sum_{j=i+1}^{d}\alpha_j},
&0\le i\le d-1,\\
1,
&i=d,\\
e^{-\frac{2\pi\ii}{N}\sum_{j=d+1}^{i}\alpha_j},
&d+1\le i\le n.
\end{cases}
\end{equation*}
Recall also that $\alpha_{n+1-j}=\alpha_j$.
This implies $\lambda_{d+r}=\lambda_{d-r}^{-1}$,
$1\le r\le d$.
Thus, the eigenvalues in the second half have the same orders as those in the first half.

Let $h:=\gcd(N,\ell_0,\ldots,\ell_{d-1})$, where $\ell_i=\sum_{j=i+1}^{d}\alpha_j$.
Then $(\zeta_N^{\ell_i})^{N/h}=1$ for all $i$.
Since $M_F$ is diagonalizable, it follows that $(M_F)^{N/h}=I$.
On the other hand, suppose $(M_F)^p=I$ for $p\in\ZZ_{>0}$.
This implies $N\mid\ell_i p$ for all $i$.
Since $h$ is the greatest common divisor of $N,\ell_0,\ldots,\ell_{d-1}$, 
there exist integers $a_0,a_1,\ldots,a_{d-1}, b$ such that $h=a_0\ell_0+\cdots+a_{d-1}\ell_{d-1} + bN$.
Multiplying this identity by $p$ and using $N\mid\ell_i p$ for all $i$, we obtain $N\mid hp$.
Since $h\mid N$, it follows that $(N/h)\mid p$.
Thus, $\ord(M_F) =\frac{N}{\gcd(N,\ell_0,\ldots,\ell_{d-1})}$.

Since $\ell_{d-1}=\alpha_d$ and
$\ell_{i-1}-\ell_i=\alpha_i$ for $1\le i\le d-1$,
we have
$\gcd(N,\ell_0,\ldots,\ell_{d-1})
=\gcd(N,\alpha_1,\ldots,\alpha_d)$.
By the definition of $N$,
we have
$\gcd(N,\alpha_1,\ldots,\alpha_d)
=\gcd(\alpha_0,\alpha_1,\ldots,\alpha_d)$.
Consequently,
$\ord(M_F)
=\frac{N}{\gcd(\alpha_0,\alpha_1,\ldots,\alpha_d)}
=\tilde{N}$.
\end{proof}

\subsection{$3 \times 3$ Example} \label{3times3-example}

In this section, we discuss in more detail the case $n=2$.
The condition $m \in \mathcal{A}_{3}$ is just $-1 \leq m_0 \leq \tfrac{1}{2}$. The parameter $m_0$ is related to the holomorphic data $(k_0,k_1,k_2)$ by
\begin{equation*}
    m_0 = \tfrac{k_1 - k_0}{k_0 + k_1 + k_2 + 3}, 
    \qquad 
    t = \tfrac{3}{k_0 + k_1 + k_2 + 3} z^{\frac{k_0 + k_1 + k_2 + 3}{3}}.
\end{equation*}
We consider the scalar ODE
\begin{equation*}
    \zeta(\partial_\zeta + m_0)\,
    \zeta \partial_\zeta\,
    \zeta(\partial_\zeta - m_0) x_0
    = -x_0,
    \qquad
    \partial_\zeta = \zeta \frac{d}{d\zeta}.
\end{equation*}
Its solution
\begin{equation*}
\small
g(\zeta)
=
\int_{c-i\infty}^{c+i\infty}
\Gamma\!\left(\frac{-m_0}{3} - t\right)
\Gamma\!\left(\frac{1}{3} - t\right)
\Gamma\!\left(\frac{m_0+2}{3} - t\right)
\,3^{-3t} e^{-3\pi i t} \zeta^{-3t}\, dt.
\end{equation*}
is holomorphic on the sector $-\tfrac{3\pi}{2} < \arg \zeta < -\tfrac{\pi}{2}$, with asymptotic behaviour
\begin{equation*} 
    g(\zeta) \sim \tfrac{4\pi^2 \ii}{\sqrt{3}} e^{\frac{1}{\zeta}} = C_{asymp} e^{\frac{1}{\zeta}}
\end{equation*}
as $\zeta \to 0$ in that sector.
Recall that $g_0(\zeta)=C_{asymp}^{-1}g(\zeta)$ and $g_k(\zeta) = \omega^{k} g_0(\omega^k \zeta)$.
By (\ref{monodromy-of-g}), we have
\begin{equation}\label{recursion-3times3} 
	g_k = sg_{k-1} -s g_{k-2}+g_{k-3}, \quad k \in \ZZ
\end{equation}
The normalized fundamental solution is
\begin{equation} 
	\hat{\chi}_g^{row} = (g_0, g_{-1}, g_{-2}),
\end{equation}
and
\begin{equation*}
	F = \begin{pmatrix}
		1 & 0 & 0 \\
		-s & 1 & 0 \\
		0 & 0 & 1
	\end{pmatrix} 
\end{equation*}
where $s = s_1 = 2\cos\!\left(\frac{2\pi}{3}(m_0+1)\right) + 1$ is the Stokes data.
In this case $s_1 = s_2$.
Then $M_F$ and $S_F$ are given by

\begin{equation*}
	M_F = 
	\begin{pmatrix}
	s & 1 & 0 \\
	 -s & 0 & 1 \\
	 1 & 0 & 0 \\
	\end{pmatrix}, \quad
	S_F =
	\begin{pmatrix}
		1 & -s & s \\
		0 & 1 & -s \\
		0 & 0 & 1
	\end{pmatrix} \quad \text{and} \quad
	(S_F)^{-T} =
	\begin{pmatrix}
		1 & 0 & 0 \\
		s & 1 & 0 \\
		s^2-s & s & 1
	\end{pmatrix}. 
\end{equation*}

A special feature of the $3 \times 3$ case is that the matrix $F$ coincides with the Stokes factor $(\tilQ_1)^{-1}$.
Therefore the formula (\ref{chione-to-chig}) implies $\hat{\chi}_g = \hat{\chi}_{1\frac{1}{3}}$. (See Table \ref{table-of-matrices}.)
We take the first row of the canonical solution, that is
$
\hat{\chi}_{1\frac{1}{3}}^{row}
=\hat{\chi}_g^{row}
=(g_0,g_{-1},g_{-2}),
$
as the initial basis of $(V,L)$,
where $L$ is defined in
Definition~\ref{the-bilinear-form}.
The set of distinguished bases is then its orbit
$
\mathcal{B}^{\mathrm{dist}}
=
\bigl(\Br_3\ltimes\{\pm1\}^3\bigr) \hat{\chi}_{1\frac{1}{3}}^{row}.
$

\begin{prop}\label{prop:distinguished-basis-3times3}
For $n=2$, the first rows of the canonical solutions
are distinguished bases with respect to the initial triangular basis
$\hat{\chi}_{1\frac{1}{3}}^{row}$:
\begin{equation*}
\{\hat{\chi}_r^{row}\}_{r\in\frac{1}{3}\ZZ}
\subset\mathcal{B}^{dist}.
\end{equation*}
More precisely, they form two $\phi$-orbits:
\begin{equation*}
	\begin{aligned}
		&\hat{\chi}_{1\frac{1}{3}+\frac{2}{3}k}^{row} =\phi^k(\hat{\chi}_{1\frac{1}{3}}^{row}) =(g_k,g_{k-1},g_{k-2}),\\
		&\hat{\chi}_{1+\frac{2}{3}k}^{row} =\phi^k(\hat{\chi}_1^{row}) =(g_{k-1},g_k-sg_{k-1},g_{k-2}),
\end{aligned}
\quad k \in \ZZ.
\end{equation*}
\end{prop}

\begin{proof}
The explicit matrices $B_r$ and the Gram matrices
in Table~1 give
\begin{equation*}
\hat{\chi}_{r+\frac{1}{3}}^{row}
=\hat{\chi}_r^{row}B_r
=
\begin{cases}
\sigma_1(\hat{\chi}_r^{row}),
&r\in1+\frac{2}{3}\ZZ,\\
\sigma_2(\hat{\chi}_r^{row}),
&r\in1\frac{1}{3}+\frac{2}{3}\ZZ.
\end{cases}
\end{equation*}
Since $\hat{\chi}_{1\frac{1}{3}}^{row}$ is the initial distinguished
basis, these relations show that every
$\hat{\chi}_r^{row}$ belongs to
$\mathcal{B}^{dist}$.

Moreover, the same explicit matrices satisfy
\begin{equation*}
B_{r+\frac{2}{3}}=B_r,
\quad
B_{1\frac{1}{3}}B_1=M_F.
\end{equation*}
Hence, for every $k\in\ZZ$,
\begin{equation*}
\hat{\chi}_{1\frac{1}{3}+\frac{2}{3}k}^{row}
=\hat{\chi}^{row}_{1\frac{1}{3}}M_F^k
=\phi^k(\hat{\chi}^{row}_{1\frac{1}{3}}).
\end{equation*}
The other family is obtained by the change of basis $B_1^{-1}$:
\begin{equation*}
\hat{\chi}_{1+\frac{2}{3}k}^{row}
=\hat{\chi}_{1\frac{1}{3}+\frac{2}{3}k}^{row}B_1^{-1}
=\phi^k(\hat{\chi}^{row}_{1\frac{1}{3}}B_1^{-1})
=\phi^k(\hat{\chi}_1^{row}).
\end{equation*}
The explicit expressions follow from
$\phi^k(g_j)=g_{j+k}$.
\end{proof}

Thus increasing the index $r$ by $\frac{1}{3}$ corresponds to
applying $\sigma_1$ or $\sigma_2$ alternately to the basis
$\hat{\chi}^{row}_r$.


\begin{table}[htbp] 
\renewcommand{\arraystretch}{1.2}
\caption{Examples of distinguished bases $\hat{\chi}_k^{row}$ and distinguished matrices $ L(\hat{\chi}^{row}_k, \hat{\chi}^{row}_k)$ for $n=2$ with respect to $(V,L,\hat{\chi}^{row}_{1\frac{1}{3}})$.}
\label{example-of-distinguished-bases/matrices}
\begin{tabular}{c|c|c} \hline
	$k$ & $\hat{\chi}_k^{row}$   & $ L(\hat{\chi}^{row}_k, \hat{\chi}^{row}_k)$ \\ \hline \hline
	$1$ & $(g_{-1}, g_0 -s g_{-1}, g_{-2})$ &
	$\begin{pmatrix}
			1 & 0 & 0 \\
			-s & 1 & 0 \\
			s & -s & 1
	\end{pmatrix}$ \\ \hline 
	$1\frac{1}{3}$ & $(g_0, g_{-1}, g_{-2} )$ &
	$\begin{pmatrix}
			1 & 0 & 0 \\
			s & 1 & 0 \\
			s^2-s & s & 1
	\end{pmatrix}$ \\ \hline 
	$1\frac{2}{3}$ &   $( g_0, g_1 -sg_{0}, g_{-1})$ &
	$\begin{pmatrix}
			1 & 0 & 0 \\
			-s & 1 & 0 \\
			s & -s & 1
	\end{pmatrix}$ \\ \hline 
	$2$ & $( g_1, g_0, g_{-1} )$ &
	$\begin{pmatrix}
			1 & 0 & 0 \\
			s & 1 & 0 \\
			s^2-s & s & 1
	\end{pmatrix}$ \\ \hline 
	$2\frac{1}{3}$ & $(g_1, g_2 - sg_1 , g_0 )$ &
	$\begin{pmatrix}
			1 & 0 & 0 \\
			-s & 1 & 0 \\
			s & -s & 1
	\end{pmatrix}$ \\ \hline 
	$2\frac{2}{3}$ & $(g_2,g_1,g_0 )$ &
	$\begin{pmatrix}
			1 & 0 & 0 \\
			s & 1 & 0 \\
			s^2-s & s & 1
	\end{pmatrix}$ \\ \hline 
\end{tabular}
\end{table}


\begin{rem}
In general, the upper triangular matrix $S_F$ is not a Stokes matrix. 
In the case of the quantum cohomology of projective spaces, this special feature can be seen in the tables 6.1 of \cite{CDG3}. 
The authors of \cite{CDG3} point out that it occurs only in the cases $n=1$ and $n=2$.
\end{rem}

\subsection{A closed formula for $L$} \label{subsec:A closed formula for L}

Throughout this subsection, we work in the case $n=2$.
We express the values of $L$ on the solutions $g_i$ in terms
of a single sequence $(a_i)_{i\in\mathbb{Z}}$.
The same sequence determines the powers of $\phi$, allowing us to
study the boundedness and periodicity of its ``cyclic orbit" through
$\hat{\chi}^{row}_{1\frac{1}{3}}=(g_0,g_{-1},g_{-2})$.

Recall from Definition \ref{def:chig} that $g_{i} = \omega^{i} g_0(\omega^{i} \zeta)$.
For simplicity, we write $L_{ij} := L(g_i, g_j)$.
As shown in Table \ref{example-of-distinguished-bases/matrices}, 
\begin{equation}
L(\hat{\chi}^{row}_{1\frac{1}{3}}, \hat{\chi}^{row}_{1\frac{1}{3}}) =
\begin{pmatrix}
1 & 0 & 0 \\
s & 1 & 0 \\
s^2 - s & s & 1
\end{pmatrix}.
\end{equation}

\begin{defi} 
For each $r \in \ZZ$, set $a_r :=L(g_0, g_r)$.
\end{defi}
Examples of this definition will appear in Proposition 4.27 below.

By (\ref{recursion-3times3}), we have the recursion formula for $a_r$
\begin{equation} \label{recursion-formula}
a_{r+3} = s a_{r+2} - s a_{r+1} + a_r,
\end{equation}
with $a_0=1$, $a_{-1}=a_{-2}=0$.
These initial values and the recursion formula determine the sequence uniquely in both directions.

\begin{lem} \label{lem:L-sequence}
For any $i,j \in \ZZ$, $L_{ij}$ satisfy
\begin{equation*}
L_{i+1,j+1} = L_{ij},
\qquad
L_{0,i} = L_{0,-i-3}.
\end{equation*}
In particular, all the values $L_{ij}$ are determined by $( a_r )_{ r \in \ZZ_{\geq -2} }$.
\end{lem}

\begin{proof}
The first identity follows from $\phi \in \Aut(V,L)$.
We prove the second identity by induction with respect to $i$.
It is easy to check that $L_{0,1} = L_{0,-4} = s$, $L_{0,0} = L_{0,-3}=1$ and $L_{0,-1} = L_{0,-2} = 0$.
Assume that $L_{0,k} = L_{0,-k-3}$ for any $-1 \leq k \leq i$.
Then we have
\begin{equation*}
	\begin{aligned}
		&L_{0,i+1} = L(g_0,g_{i+1}) = L(g_0, s g_i - s g_{i-1} + g_{i-2}) = sL_{0,i} - sL_{0,i-1} + L_{0,i-2} \\
		&= sL_{0,-i-3} - sL_{0,-i-2} + L_{0,-i-1} = L(g_0, s g_{-i-3} - s g_{-i-2} + g_{-i-1}) = L(g_0, g_{-i-4}) \\
		&= L_{0,-i-4}.
	\end{aligned}
\end{equation*}
\end{proof}

We next relate this sequence to the powers of $\phi$.
Recall that $M_F$ represents $\phi$ with respect to $\hat{\chi}^{row}_{1\frac{1}{3}}$.

\begin{prop} \label{prop:MF-and-am}
For every $r \in \ZZ$,
\begin{equation} \label{eq:MF-and-am}
	M_F^r = 
	\begin{pmatrix}
		a_r & a_{r-1} & a_{r-2} \\
		a_{r-2} - sa_{r-1} & a_{r-3} - sa_{r-2} & a_{r-4} - sa_{r-3} \\
		a_{r-1} & a_{r-2} & a_{r-3}
	\end{pmatrix}.
\end{equation}
\end{prop}

\begin{proof}
Let $\varepsilon_1,\varepsilon_2,\varepsilon_3$ denote the standard basis vectors of $\RR^3$ and
\begin{equation*}
	u_r :=
	M_F^r \varepsilon_1 =
	\begin{pmatrix}
		x_r \\
		y_r \\
		z_r
	\end{pmatrix}.
\end{equation*}
Since
\begin{equation*}
	M_F
	\begin{pmatrix}
		x \\
		y \\
		z \\
	\end{pmatrix}
	= 
	\begin{pmatrix}
		sx + y\\
		-sx + z \\
		x
	\end{pmatrix},
\end{equation*}
we have 
\begin{align}
	 x_{r+1} = s x_r + y_r, \quad
	 y_{r+1} = -s x_r + z_r, \quad
	 z_{r+1} = x_r. 
\end{align}
Eliminating $y_r$ and $z_r$ from these relations, we obtain
\begin{equation}
	x_{r+3} = s x_{r+2} -s x_{r+1} + x_r.
\end{equation}
Moreover, $x_0=1=a_0$, $x_1=s=a_1$, and $x_2=s^2-s=a_2$.
Since the coefficient of $x_r$ in the recurrence is nonzero, the recurrence determines the sequence both forward and backward.
Thus $x_r=a_r$ for all $r \in \ZZ$. Consequently,
\begin{equation*}
z_r=a_{r-1},\qquad
y_r=-sx_{r-1}+z_{r-1}=a_{r-2}-sa_{r-1},
\end{equation*}
and therefore
\begin{equation}
	u_r
	=
	\begin{pmatrix}
		a_r \\
		a_{r-2}-sa_{r-1} \\
		a_{r-1}
	\end{pmatrix}.
	\label{eq:u-m}
\end{equation}
Since $M_F \varepsilon_2=\varepsilon_1$ and $M_F \varepsilon_3=\varepsilon_2$,
we have
\begin{equation*}
	M_F^r \varepsilon_2=u_{r-1},
	\quad
	M_F^r \varepsilon_3=u_{r-2}.
\end{equation*}
Therefore, the columns of $(M_F)^m$ are respectively
$u_r,u_{r-1}$, and $u_{r-2}$. Using \eqref{eq:u-m}, we obtained the desired formula.
\end{proof}

\begin{cor} \label{cor:order=period}
The automorphism $\phi$ has finite order if and only if $(a_r)_{r \in \ZZ}$ is periodic. In this case,
\begin{align*}
	\ord(\phi) = \per( (a_r)_{r \in \ZZ}).
\end{align*}
Moreover, the sequence $( a_r )_{r \in \ZZ}$ is bounded if and only if the cyclic orbit
\begin{align*}
	\langle \phi \rangle (\hat{\chi}^{row}_{1\frac{1}{3}}) = \{ \phi^r(\hat{\chi}^{row}_{1\frac{1}{3}}) \, | \, r \in \ZZ \}
\end{align*}
is bounded.
\end{cor}

\begin{proof}
The first assertion follows immediately from \eqref{eq:MF-and-am}.
For boundedness, we identify each ordered basis $\underline b$ of $V$ with the unique invertible matrix $B$ satisfying $\underline{b}=\hat{\chi}^{row}_{1\frac{1}{3}} B$ 
and use the maximum-entry norm on matrices.
If $|a_r|\leq K$ for all $r \in \ZZ$, then
\eqref{eq:MF-and-am} gives
\begin{equation*}
\|M_F^r\|_{\max}\leq (1+|s|)K 
\end{equation*}
for all $r \in \ZZ$.
Conversely, the boundedness of the sequence follows from $|a_r|\leq\|M_F^r\|_{\max}$.
Thus $(a_r)$ is bounded if and only if $\{M_F^r\}_{r \in \ZZ}$ is bounded.
\end{proof}

\begin{prop} For every $r \geq 0$,
\begin{equation} 
	a_{r} 
	=  \sum_{k=0}^{r} U_{k}\left(\frac{s-1}{2}\right), \label{Chebyshev} 
\end{equation}
where $U_{k}$ is the Chebyshev polynomial of the second kind. 
If $-1 < s < 3$, we have 
\begin{equation}\label{eq:Chebyshev-2}
a_{r}=
\frac{\cos(\theta/2)-\cos((r+\tfrac{3}{2})\theta)}
{2\sin\theta\sin(\theta/2)}
\end{equation}
where $s = 1 + 2 \cos \theta$ with $0 < \theta < \pi$.
\end{prop}

\begin{proof}
Set $b_r=a_r-a_{r-1}$ for $r\in\mathbb{Z}$.
By \eqref{recursion-formula},
\begin{equation*}
b_{r+2}-(s-1)b_{r+1}+b_r =a_{r+2}-sa_{r+1}+sa_r-a_{r-1} =0.
\end{equation*}
We have $b_{-1}=0$, $b_0=1$ and $b_1=s-1$.
The Chebyshev polynomial of the second kind $U(x)$ satisfies
\begin{equation*}
U_0(x)=1,\qquad U_1(x)=2x,\qquad
U_{r+2}(x)=2xU_{r+1}(x)-U_r(x).
\end{equation*}
Comparing the recursion formula and initial values gives
\begin{equation*}
b_r=U_r\!\left(\frac{s-1}{2}\right),
\qquad r\geq 0.
\end{equation*}
Summing from $0$ to $r$ and using $a_{-1}=0$ proves
\eqref{Chebyshev}.

Suppose $-1<s<3$, and write $s=1+2\cos\theta$ with
$0<\theta<\pi$. Since
\begin{equation*}
U_k(\cos\theta)=\frac{\sin((k+1)\theta)}{\sin\theta},
\end{equation*}
we have, for $r\geq 0$,
\begin{equation*}
a_r=\frac{1}{\sin\theta}\sum_{k=1}^{r+1}\sin(k\theta).
\end{equation*}
The identity
\begin{equation*}
2\sin(\theta/2)\sin(k\theta)
=\cos((k-\tfrac12)\theta)-\cos((k+\tfrac12)\theta)
\end{equation*}
gives \eqref{eq:Chebyshev-2} for $r\geq 0$.

The right-hand side of \eqref{eq:Chebyshev-2} vanishes for $r=-1,-2$ and is unchanged under $r\mapsto-r-3$.
This extends the formula to all $r\in\mathbb{Z}$.
\end{proof}

The explicit formula gives the following consequences for the
sequence and the cyclic orbit of $\phi$.

\begin{theo}\label{thm:an-bounded-periodic}
Suppose that $-1<s<3$, and let $0<\theta<\pi$ be determined by
$s=1+2\cos\theta$.
\begin{enumerate}
\item The sequence $(a_r)_{r \in \ZZ}$ and the cyclic orbit
$\langle\phi\rangle\hat{\chi}^{row}_{1\frac{1}{3}}$ are bounded.
\item The sequence is periodic if and only if
$\theta/(2\pi)\in\mathbb{Q}$.
If $\theta/(2\pi)=p/q$ in lowest terms, with $q>0$, then
\begin{equation*}
\operatorname{period} \left((a_r)_{r \in \ZZ } \right)
=\operatorname{order}(\phi)=q.
\end{equation*}
In particular, the cyclic orbit is finite if and only if
$\theta/(2\pi)\in\mathbb{Q}$.
\end{enumerate}
\end{theo}

\begin{proof}
Since $0<\theta<\pi$, formula \eqref{eq:Chebyshev-2} gives
\begin{equation*}
|a_r| = \left\lvert \frac{ \sin \left( \frac{(r+2)\theta}{2} \right) \sin \left( \frac{(r+1)\theta}{2} \right) }{\sin \theta \sin (\theta /2)}  \right\rvert \leq\frac{1}{\sin\theta\sin(\theta/2)}
\qquad (r\in\mathbb{Z}).
\end{equation*}
Hence the sequence is bounded, and boundedness of the cyclic orbit follows from Corollary \ref{cor:order=period}.

To prove the second assertion, let $T$ be a positive integer.
If $T\theta\in2\pi\mathbb{Z}$, then \eqref{eq:Chebyshev-2} gives $a_{r+T}=a_r$ for every $r\in\mathbb{Z}$.
Conversely, suppose $(a_r)$ satisfies $a_{r+T}=a_r$ for every $r\in\mathbb{Z}$.
Then $b_r = a_r-a_{r-1}$ also satisfies $b_{r+T}=b_r$ for every $r\in\mathbb{Z}$.
Taking the difference of the expressions in
\eqref{eq:Chebyshev-2}, we obtain
\begin{equation*}
b_r = \frac{2 \sin((r+1)\theta) \sin(\theta /2) }{2 \sin \theta \sin (\theta/2)} =\frac{\sin((r+1)\theta)}{\sin\theta}
\quad (r\in\mathbb{Z}).
\end{equation*}
Thus $b_{T-1}=b_{-1}=0$ and $b_T=b_0=1$ imply
\begin{equation*}
\sin(T\theta)=0,\quad \sin((T+1)\theta)=\sin\theta.
\end{equation*}
Using the addition formula in the second equality and
$\sin\theta\neq 0$, we get $\cos(T\theta)=1$.
Therefore, $a_{r+T}=a_r$ for every $r\in\mathbb{Z}$ if and only if $T\theta\in2\pi\mathbb{Z}$.
A positive integer $T$ satisfying this condition exists if and only if $\theta/(2\pi)$ is rational.
If $\theta/(2\pi)=p/q$ in lowest terms with $q>0$, the condition is equivalent to $q\mid T$. 
Consequently, the least positive period is $q$, and Corollary \ref{cor:order=period} gives $\operatorname{ord}(\phi)=q$.
\end{proof}

\begin{rem}
In terms of the asymptotic data $-1<m_0<\tfrac12$, we have
$\theta=\tfrac{2\pi}{3}(m_0+1)$.
Thus, when $m_0$ is rational,
\begin{equation*}
\per \left( (a_r)_{r \in \ZZ} \right)
=\operatorname{den}\!\left(\frac{m_0+1}{3}\right),
\end{equation*}
where $\operatorname{den}(r)$ denotes the positive denominator of a rational number $r$ in lowest terms.
For the solutions corresponding to $k_0,k_1,k_2\in\mathbb{Z}_{\geq 0}$, Theorem \ref{periodicity-of-the-braid} and Corollary \ref{cor:order=period} give
\begin{equation*}
\tilde N=
\frac{N}{\gcd(\alpha_0,\alpha_1,\alpha_2)},
\quad
N=\alpha_0+\alpha_1+\alpha_2,\quad \alpha_i=k_i+1,
\end{equation*}
as a period of the sequence.
\end{rem}

\begin{table}[htbp]
\caption{Sequences $(a_r)_{r\geq0}$ and their behaviour, signatures of $L+L^T$,
and cardinalities of the sets of distinguished Stokes matrices and
distinguished bases modulo sign changes for $s=-1,0,1,2,3$.
For $s=0,1,2$, the displayed tuple is one repeating block,
and the indicated period is the least period of the sequence.}
\label{tab:sequence-signature-orbits}
\centering
\small
\setlength{\tabcolsep}{3pt}
\renewcommand{\arraystretch}{1.4}
\begin{tabular}{c|c|c|c|c|c}
\hline
$s$
& $(a_r)_{r\geq0}$
& Behaviour
& $\operatorname{sign}(L+L^T)$
& $\#(\mathcal{S}^{\mathrm{dist}}/\{\pm1\}^3)$
& $\#(\mathcal{B}^{\mathrm{dist}}/\{\pm1\}^3)$
\\ \hline \hline
$-1$
& $a_r=\dfrac{1+(-1)^r(2r+3)}{4}$
& \begin{tabular}[c]{@{}c@{}}
  Unbounded;\\linear growth
  \end{tabular}
& $(+,+,0)$ & $4$ & $\infty$
\\ \hline
$0$
& $(1,0,0)$
& Period $3$
& $(+,+,+)$ & $1$ & $6$
\\ \hline
$1$
& $(1,1,0,0)$
& Period $4$
& $(+,+,+)$ & $4$ & $16$
\\ \hline
$2$
& $(1,2,2,1,0,0)$
& Period $6$
& $(+,0,0)$ & $1$ & $\infty$
\\ \hline
$3$
& $a_r=\dfrac{(r+1)(r+2)}{2}$
& \begin{tabular}[c]{@{}c@{}}
  Unbounded;\\quadratic growth
  \end{tabular}
& $(+,-,-)$ & $\infty$ & $\infty$
\\ \hline
\end{tabular}
\end{table}
Finally, we summarize in Table 3 some of the properties of L and the distinguished bases in the integer Stokes cases. 
The signatures in Table~\ref{tab:sequence-signature-orbits} are obtained directly from the symmetric Gram matrix
\begin{equation*}
(L+L^T)(\hat{\chi}^{row}_{1\frac{1}{3}},\hat{\chi}^{row}_{1\frac{1}{3}})
=
\begin{pmatrix}
2 & s & s^2-s \\
s & 2 & s \\
s^2-s & s & 2
\end{pmatrix}.
\end{equation*}
For the integral values of $s$ considered here, the braid group action preserves the lattice $\ZZ g_0 \oplus \ZZ g_{-1} \oplus \ZZ g_{-2}$.
Every vector $v$ in a distinguished basis satisfies $(L+L^T)(v,v)=2$.
When $s=0,1$, the symmetric form is positive definite, so this lattice contains only finitely many such vectors.
Consequently, both the set of distinguished bases and the set of distinguished matrices are finite.

The cardinalities modulo sign changes follow from \cite{HL}. For example, Theorem 4.13 and Example 7.16.
For $s=-1,2$, the symmetric form is positive semidefinite, and the set of distinguished matrices is finite whereas the set of distinguished bases is infinite.
These cases show that finiteness of the matrix orbit does not imply finiteness of the basis orbit.
In particular, for $s=2$, the cyclic orbit $\langle\phi\rangle\hat{\chi}^{row}_{1\frac{1}{3}}$ has six elements, although the full braid group orbit of distinguished bases is infinite.

\section*{Acknowledgments}
The author would like to thank Prof. Martin Guest for many helpful discussions and valuable comments on this work, and Prof. Claus Hertling for his helpful suggestions and comments on the braid group action. This paper is a part of the outcome of research performed under a Waseda University Grant for Special Research Projects (Project number: 2025C-720).

\bibliographystyle{plain}
\bibliography{ref}

\bigskip

\begin{flushleft}
\textit{Department of Applied Mathematics}\\ 
\textit{Faculty of Science and Engineering}\\
\textit{Waseda University}\\
\textit{3-4-1 Okubo, Shinjuku-ku, Tokyo 169-8555}\\
\textit{JAPAN}\\
\end{flushleft}

\end{document}